\documentclass[a4paper,11pt]{article}
\usepackage[T1]{fontenc}
\usepackage[utf8]{inputenc}
\usepackage[english]{babel}
\usepackage{amsfonts}
\usepackage{amsthm}
\usepackage{amsmath}
\usepackage{amssymb}
\usepackage{amsthm}
\usepackage{graphicx,subfigure}
\usepackage{color}
\usepackage{mathrsfs}
\usepackage{algorithm}
\usepackage{algpseudocode}
\usepackage{cite}
\newtheorem{theorem}{Theorem} 
 
\newtheorem{remark}{Remark}

\def\R{\mathbb{R}}
\def\E{\mathbb{E}}
\def\V{{\rm \mathbb{V}ar}}
\def\C{{\rm \mathbb{C}ov}}
\def\be{\begin{equation}}
\def\ee{\end{equation}}

\algdef{SE}{Begin}{End}{\textbf{begin}}{\textbf{end}}

\usepackage{hyperref}

\begin{document}
\title{Reduced variance random batch methods\\ for nonlocal PDEs}
\author{Lorenzo Pareschi \\
		{\small 	Maxwell Institute for Mathematical Sciences \&
Department of Mathematics} \\
		{\small Heriot-Watt University, Edinburgh, UK} \\
		{\small	Department of Mathematics and Computer Science} \\
		{\small University of Ferrara, Italy} \\
		{\small\tt l.pareschi@hw.ac.uk} \\[5mm]	
		Mattia Zanella \\
		{\small	Department of Mathematics } \\
		{\small University of Pavia, Italy} \\
		{\small\tt mattia.zanella@unipv.it}}
\date{\it Dedicated to our dear friend and colleague Shi on his 60th birthday}

\maketitle

\begin{abstract}
Random Batch Methods (RBM) for mean-field interacting particle systems enable the reduction of the quadratic computational cost associated with particle interactions to a near-linear cost. The essence of these algorithms lies in the random partitioning of the particle ensemble into smaller batches at each time step. The interaction of each particle within these batches is then evolved until the subsequent time step. This approach effectively decreases the computational cost by an order of magnitude while increasing the amount of fluctuations due to the random partitioning. In this work, we propose a variance reduction technique for RBM applied to nonlocal PDEs of Fokker-Planck type based on a control variate strategy. The core idea is to construct a surrogate model that can be computed on the full set of particles at a linear cost while maintaining enough correlations with the original particle dynamics. Examples from models of collective behavior in opinion spreading and swarming dynamics demonstrate the great potential of the present approach.
\medskip


\end{abstract}

\tableofcontents

\section{Introduction}

Meanfield equations are derived as reduced complexity models describing the dynamics of large systems of interacting particles^^>\cite{Ja}. Typically, the trajectories of systems of particles are governed by stochastic differential equations for the position $x_i \in \mathbb R^d$ and velocity $v_i \in \mathbb R^d$ of the $i$th particle and are translated in systems of stochastic differential equations of the form 
\begin{equation}
\label{eq:system}
\begin{cases}
dx_i = v_i dt, \\
dv_i = \left(\dfrac{1}{N}\displaystyle \sum_{j=1}^N P(x_i,x_j,v_i,v_j)(v_j-v_i)\right)dt + \sqrt{2D(v_i)^2}dW_i^t, \qquad i = 1, \dots, N, 
\end{cases}
\end{equation}
where $P(\cdot)\ge0$ is a suitable interaction function. In the limit of a large number of particles $N$, the individual trajectories become irrelevant, and we have the opportunity to describe the average behavior of the system. In such a situation, the strength of the individual interactions becomes small, and each particle feels the effect of the force field created by all the other particles. Therefore, the statistical behavior of the system of particles is obtained by studying the evolution of the particles' density $f(x,v,t)$ which obeys a nonlocal PDE of Fokker-Planck type
\begin{equation}
\label{eq:VFP}
\partial_t + v\cdot \nabla_x f  = \nabla_v \cdot \left[ \mathcal B[f](x,v,t)f(x,v,t) \right] +   \Delta_v (D^2(v)f),
\end{equation}
where the term $D(v):\mathbb R^d\to \mathbb R$ expresses the local relevance of the diffusion and the nonlocal drift term $\mathcal B[f](x,v,t)$ is defined as follows
\begin{equation*}
\mathcal B[f](x,v,t)=  \int_{\mathbb R^d \times \mathbb R^d} P(x,y,v,w)(v-w) f(y,w,t)dv\, dy.
\end{equation*}
The derivation of the mean-field model \eqref{eq:VFP} is essential to investigate mathematically the long time behavior and the formation of collective structures and self-organisation features of the multi-agent system \eqref{eq:system}. Indeed, since the trajectory of each particle is influenced by all the other particles, the computational complexity of the particle model is $O(N^2)$ and makes the microscopic approach rapidly infeasible for $N\gg0$. 

In this direction, Random Batch Methods (RBM) have gained increased popularity in the community of interacting particle systems since they are capable to efficiently reduce the computational complexity to $O(MN)$, where $M$ is the size of the batch, at the cost of introducing an additional error in the dynamics^^>\cite{CJT, JL, JLLq, KZ}. In the context of many-particle systems, RBM methods have been proposed in^^>\cite{JLL}. Earlier approaches based on analogous Monte Carlo strategies were developed in^^>\cite{AP} and^^>\cite{CPZ,CZ} in the presence of uncertain quantities. The core of these algorithms is to randomly divide the set of particles into smaller batches at each time step, a technique closely resembling the batch methods of stochastic gradient descent algorithms. The interaction of each particle within these batches is then evolved until the subsequent time step. This approach significantly reduces the computational cost by an order of magnitude, albeit at the expense of heightened fluctuations due to the random partitioning.

In contrast with random search algorithms were the increase of fluctuations may also have a positive impact on the overall process, in RBM methods, since we aim at computing as accurate as possible the $O(N^2)$ summation, reducing the variance of the batch algorithm is of paramount importance and represents one of the main challenges.  
In this paper, inspired by similar ideas in the field of uncertainty quantification for kinetic and mean-field equations^^>\cite{DP,DP2020,Pa21}, we propose a control variate strategy based on a simpler surrogate model to reduce batch variance. However, compared to uncertainty quantification, where due to the static nature of stochastic samples it is easy to obtain correlations, here the samples, although initially correlated, naturally decorrelate as a consequence of their temporal evolution. For this purpose, we employ a reduced-complexity surrogate that can be calculated at the cost of $O(N)$, sharing the same asymptotic state as the original model.

We show that the present strategy permits obtaining a significant reduction in the variance of the batch, leading to a further speed-up in computational cost for a given accuracy. This is illustrated using models of collective behavior in opinion spreading and swarming dynamics.

The rest of the manuscript is organized as follows. In Section 2, we revise some nonlocal Fokker-Planck models that will be used as prototype examples to illustrate our approach and to test the numerical schemes. Next, Section 3 is devoted to presenting the reduced variance RBM methodology. Several numerical examples are then given in Section 4. The last section contains some concluding remarks.

\section{Vlasov-Fokker-Planck models}\label{sect:2}
In this section we survey some examples of multi-agent systems whose mean-field limit corresponds to a Vlasov-Fokker-Planck model of the type \eqref{eq:VFP}. First, we define the classical model for transport-aggregation-diffusion. Hence, we discuss the recent extensions of this model to kinetic alignment models in $\mathbb R^d$. In the space homogeneous setting, we also survey recent examples presented in opinion dynamics that have attracted great interest in the mathematical community studying systems of interacting particles. 

\subsection{The classical Vlasov-Fokker-Planck equation} 
Fokker-Planck-type collision operators are commonly introduced to define reduced complexity models for collisional dynamics, see^^>\cite{T0,Villani}.  The classical Vlasov-Fokker-Planck equation can be obtained from \eqref{eq:VFP} with the choice 
\[
D^2(v) = \sigma^2 \in \mathbb R_+, \qquad P(x,y,v,w)\equiv 1,
\]
such that 
\[
\mathcal B[f](x,v,t) = v-u,
\]
where
\[
u = \int_{\mathbb R^{2d}}vf(x,v,t)dv\,dx, \qquad \sigma^2 = \dfrac{1}{d}\int_{\mathbb R^{2d}}|v-u|^2f(x,v,t)dv\,dx,
\]
are the conserved global mean and temperature respectively. We may express the global Maxwellian as follows
\[
f^{\infty}(x,v) =  \left(\dfrac{1}{2\pi\sigma^2} \right)^{d/2}\exp\left\{-\dfrac{|v-u|^2}{2\sigma^2}\right\},
\] 
where we assumed that the total mass is normalized to one. The study of equilibration rates on the velocity space, given a initial distribution $f_0$, is a key problem in kinetic theory, see e.g.^^>\cite{BR,CT,T0}. The trends to equilibrium are  determined in terms of the Boltzmann H-functional 
\[
H[f](t) = \int_{\mathbb R^{2d}} f\log f dv\,dx. 
\] 
In particular, starting from an initial distribution with finite entropy, $f(x,v,t)$ converges in relative entropy to the equilibrium solution exponentially fast. 

\subsection{Vlasov-Fokker-Planck models in collective dynamics} 
\label{sect:swarming}
Vlasov-Fokker-Planck-type equations have been used more recently to understand long-rage interactions in collective phenomena. In this direction, a key role has been played by the so-called swarming dynamics in which large systems of agents tend to share the same velocity for large times, see^^>\cite{CS}. The next example focuses on a kinetic equation for swarming dynamics. The kinetic equation can be obtained from \eqref{eq:VFP} with the choice 
\[
P(x,y,v,w) = \dfrac{K(v-w)}{\left(\xi^2 + |x-y|^2\right)^\beta}, 
\]
such that the nonlocal drift assumes the form 
\[
\mathcal B[f](x,v,t) = \int_{\mathbb R^{2d}}\dfrac{K(v-w)}{\left(\xi^2 + |x-y|^2\right)^\beta}f(y,w,t)dw\,dy.
\]
The resulting PDE corresponds to the mean-field limit of the well-known Cucker-Smale model^^>\cite{HL,HT,CFRT,CFTV,ST}. Several results provide conditions for the asymptotic alignment of velocities under suitable assumptions on the coefficients of the model. In particular, if $\beta<1/2$ we have the large time collapse of the support of the emerging Maxwellian distribution and therefore asymptotic alignment of the agents' velocities^^>\cite{FK}.  

The emergence of collective structures by means of space-dependent interaction forces can be also observed in multi-agent models for milling^^>\cite{DO,CKR} and synchronization^^>\cite{CHY}. 

\subsection{Alignment models in consensus dynamics}\label{sect:2.3}
The mathematical modelling of consensus formation has seen a growing interest in recent years^^>\cite{GP}. Classical methods of statistical mechanics served as a powerful basis to model interactions  in cooperative multi-agent systems and to reduce the complexity linked to interaction forces.   
The study of alignment dynamics has produced a heterogeneity of models, see^^>\cite{APZ17,DW,FPTT,GP,HK,MT,T} and the references therein. More recently these approaches have found interesting applications in optimization and data science^^>\cite{CBO,HPV}.

We consider the bounded confidence model for opinion formation
\begin{equation}
\label{eq:BCMF}
\partial_t f({v},t) = \partial_{{v}} \left[ \int_{[-1,1]}P(| v- w|)({v}-{w})f({w},t)d{w}f({v},t)\right], \quad {v} \in  [-1,1],
\end{equation}
where
\begin{equation}
\label{eq:BCinter}
P(| v- w|) =  \chi(| v- w|\le \delta)
\end{equation}
being $\chi(\cdot)$ the indicator function and being $\delta >0$ a given threshold. The large time behaviour of the model is a sum of Dirac delta distribution which express clusterization of the society towards a finite number of asymptotic opinions. 

The corresponding version in presence of noise is given by 
\begin{equation}
\label{eq:BCstoch}
\partial_t f({v},t) = \partial_{{v}} \left[ \int_{[-1,1]}P(| v- w|)({v}-{w})f({w},t)d{w}f({v},t)\right] + \sigma^2 \partial_v^2 (D^2(v) f(v,t)), 
\end{equation}
where $v \in [-1,1]$ and where $D(v) = \sqrt{1-v^2}$. In the simplified scenario where $\delta = 2$ we get $P\equiv 1$ and we may compute the large time behaviour of the model which corresponds to the Beta distribution
\[
f^\infty(v) = \dfrac{(1+v)^{\frac{1+m}{\sigma^2}-1} (1-v)^{\frac{1-m}{\sigma^2}-1}}{2^{\frac{2}{\sigma^2}-1}\textrm{B}\left(\frac{1+m}{\sigma^2},\frac{1-m}{\sigma^2} \right)},
\]
where $m = \int_{-1}^1vf(v,0)dv$ is the conserved mean opinion of the system. If $0<\delta<2$ the explicit computation of the asymptotic state is difficult to obtain. A possible strategy to get some insight on the large time trends is based on the construction of structure preserving numerical approximations \cite{PZ0}. 

\section{Reduced variance RBM}
Let $f = f(x,v,t): \mathbb R^{d}\times \mathbb R^{d} \times \mathbb R_+\to \mathbb R_+$ be a distribution function whose evolution is given by the Vlasov-Fokker-Planck model with nonlocal flux defined in \eqref{eq:VFP} with a constant diffusion $D(v)=\sigma$. The evolution of observable quantities is obtained by rewriting the equation in weak form. 

Let $\varphi(\cdot,\cdot)$ be a test function, then the Vlasov-Fokker-Planck model defined in \eqref{eq:VFP} can be written in weak form as follows
\[
\begin{split}
&\dfrac{d}{dt}\int_{\mathbb R^{2d} } \varphi(x,v)f(x,v,t)dv\,dx + \int_{\mathbb R^{2d}} v\cdot \nabla_x\varphi(x,v) \, f(x,v,t)dv\,dx  \\
&\quad =  \int_{\mathbb R^{2d}  } \varphi(x,v)\nabla_v\cdot \left[ \mathcal B[f](x,v,t)f(x,v,t) + \sigma^2 \nabla_v f(x,v,t)\right]dv\,dx.
\end{split}
\]
Hence, by choosing $\varphi(x,v) \equiv 1$ we may observe that the total mass is conserved since
\[
\dfrac{d}{dt} \int_{\mathbb R^{2d}}f(x,v,t)dv\,dx  = 0.
\] 
The positivity for all $t\ge0$ of the kinetic distribution $f(x,v,t)$ may be proven under suitable assumptions on the drift term and the diffusion function. Furthermore, for $\varphi(v) = v$ we have that the evolution of the mean velocity is given by
\[
\dfrac{d}{dt} \int_{\mathbb R^{2d}}vf(x,v,t)dv \,dx=  \int_{\mathbb R^{2d}} \mathcal B[f](x,v,t)f(x,v,t)dv\,dx.
\]
Therefore, the mean velocity is conserved if $P$ is a symmetric function in both $x\in \mathbb R^{d}$ and $v \in \mathbb R^{d}$. 

Given a set of sample positions and velocities $\{(x_i,v_i)\}_{i=1}^N$ at time $t\ge0$, the kinetic distribution $f(x,v,t)$ is recovered from the empirical density function 
\begin{equation}
\label{eq:fN}
f_N(x,v,t) = \dfrac{1}{N} \sum_{i=1}^N \delta(x-x_i(t))\otimes \delta(v-v_i(t)), 
\end{equation}
where $\delta(\cdot)$ is the Dirac delta function and the particles evolution follows \eqref{eq:system}. Therefore, for any test function $\varphi(x,v)$, if we denote by 
\[
(\varphi,f)(t) = \int_{\mathbb R^d} \varphi(x,v)f(x,v,t)dv,
\]
we have
\[
(\varphi,f_N)(t) = \dfrac{1}{N} \sum_{i=1}^N \varphi(x_i(t),v_i(t)).
\]
If $f(x,v,t)$ is a probability density function $\int_{\mathbb R^{2d}}f(x,v,t)dv\,dx$, we have 
\[
(\varphi,f) = \mathbb E_{X,V}[\varphi],
\]
being $\mathbb E_{X,V}[\varphi]$ the expectation of the observable quantity $\varphi$.
 The following result holds^^>\cite{Caf}
\begin{theorem}\label{th:1}
The root-mean square error is such that for each $t\ge0$
\[
\mathbb E_{X,V}\left[\left|(\varphi,f)-(\varphi,f_N) \right|^2\right]^{1/2} = \dfrac{\sigma_{\varphi}}{N^{1/2}},
\]
where $\sigma_\varphi^2 = \textrm{Var}(\varphi)$ and
\[
\textrm{Var}(\varphi) = \int_{\mathbb R^{2d}}(\varphi(x,v)-(\varphi,f))^2 f(x,v,t)dv\,dx.
\]
\end{theorem}

Since $f$ is a probability density, we observe that the nonlocal drift operator may be thought as an expectation with respect to the kinetic distribution $f(x,v,t)$ and in particular as the evaluation of the expected quantity of $P(x,y,v,w)(w-v)$ with respect to the pair $(y,w)$. Hence, we have
\[
\nabla_v \cdot \left(\mathcal B[f](x,v,t)f(x,v,t)\right) = \nabla_v \cdot\left( \mathbb E_f[P(x,\cdot,v,\cdot)(v-\cdot)]f(x,v,t)\right),
\]
where $\mathbb E_f[\cdot]$ is the expectation with respect to the probability density $f$ and is defined as 
\[
\mathbb E_f[g(x,y,v,w,t)] =  \int_{\mathbb R^{d}\times \mathbb R^{d}} g(x,y,v,w,t) f(y,w,t) dy\,dw.
\]
Therefore, the initial model in \eqref{eq:VFP} is equivalently written as 
\[
\partial_t f + v\cdot \nabla_x f = \nabla_v \cdot\left[ \mathbb E_f[P(x,\cdot,v,\cdot)(v-\cdot)]f + \sigma^2 \nabla_v f\right]
\]

\subsection{Standard RBM method}

Equation \eqref{eq:VFP} can be obtained as the mean-field limit of the particle system \eqref{eq:system}. Let us consider the empirical distribution \eqref{eq:fN} and rewrite 
the second equation in \eqref{eq:system} as  follows
\[
d v_i = \mathbb E_{f^N}[ P(x_i,\cdot,v_i,\cdot)(\cdot-v_i)]dt + \sqrt{2\sigma^2} dW_i^t,
\] 
being 
\[
 \mathbb E_{f^N}[P(x_i,\cdot,v_i,\cdot)(\cdot-v_i)] = \dfrac{1}{N} \sum_{j=1}^N P(x_i,x_j,v_i,v_j)(v_j-v_i).
\]
For any $N\gg0$ the we may observe that
\[
\mathbb E_f[P(x,\cdot,v,\cdot)(v-\cdot)] \approx \mathbb E_{f^N}[P(x_i,\cdot,v_i,\cdot)(\cdot-v_i)].
\]
As a result, we obtain  a particle dynamics with complexity $O(N^2)$. The random batch idea is based on the use of subsets $S_M$ of samples of size $M \ll N$ sampled uniformly from the set of indices $\{1,\dots,N\}$ to further approximate
\be
\label{eq:sumRBM}
\E_{f^N}[P(x_i,\cdot,v_i,\cdot)(\cdot-v_i)] \approx \frac1{M} \sum_{j\in S_M} P(x_i,x_j,v_i,v_j)(v_j-v_i),
\ee  
for all $i=1,\ldots,M$. This fast summation strategy has an overall cost of $O(MN)$ instead of $O(N^2)$ and therefore substantially alleviate computational burden linked to the particle approach. In the following, we concentrate on the dynamics of the particle system given by the following system of SDEs for $(x_i,v_i)$
\begin{equation}
\label{eq:system_RBM}
\begin{cases}
dx_i = v_idt, \\
dv_i = \dfrac{1}{M} \displaystyle\sum_{j \in S_M} P(x_i,x_j,v_i,v_j)(v_j-v_i)dt + \sqrt{2\sigma^2}dW_i^t
\end{cases}
\end{equation}
such that the empirical measure
\[
{f}_N^M(x,v,t) = \dfrac{1}{N} \sum_{i=1}^N \delta(x-x_i(t))\otimes \delta(v-v_i(t)) 
\]
converges towards $f(x,v,t)$ solution to \eqref{eq:VFP}.  Indeed, the following result can be established

\begin{theorem}\label{th:RBM}
The root-mean square error of the RBM method is such that for each $t\ge0$
\[
\mathbb E_{X,V}\left[ \left|(\varphi,f) - (\varphi,f_N^M)\right|^2 \right]^{1/2} \le \dfrac{\sigma_\varphi}{\sqrt{N}} +  \theta_\varphi\sqrt{\dfrac{1}{M}-\dfrac{1}{N}},
\]
where 
\[
 \sigma_\varphi^2 = \int_{\mathbb R^{2d}}(\varphi(x,v)-(\varphi,f))^2 f(x,v,t)dv\,dx
\]
and 
\[
 \theta_\varphi^2 = \int_{\mathbb R^{2d}}(\varphi(x,v)-(\varphi,f_N))^2 f_N(x,v,t)dv\,dx
\]
\end{theorem}
\begin{proof}
For any test function $\varphi$ we have
\[
\begin{split}
&\mathbb E_{X,V}\left[\left| (\varphi,f)-(\varphi,f_N^M) \right|^2\right]\le\\
&\quad   2\mathbb E_{X,V}\left[\left| (\varphi,f)-(\varphi,f_N) \right|^2 \right] +  2\mathbb E_{X,V}\left[\left| (\varphi,f_N)-(\varphi,f_N^M) \right|^2 \right].  \\
\end{split}\]
Hence, we may pursue the same strategy presented in \cite{CZ} to get the following estimate
\[
\mathbb E_{X,V}\left[\left| (\varphi,f_N)-(\varphi,f_N^M) \right|^2 \right]^{1/2} \le  \theta_\varphi \sqrt{\dfrac{1}{M}-\dfrac{1}{N}},
\]
from which we get thanks to Theorem \ref{th:1} 
\[
\mathbb E_{X,V}\left[\left| (\varphi,f)-(\varphi,f_N^M) \right|^2\right]^{1/2}
\le \sqrt{2}\dfrac{\sigma_\varphi}{\sqrt{N}} + \sqrt{2}\theta_\varphi \sqrt{\dfrac{1}{M}-\dfrac{1}{N}}.
\]
\end{proof}
From the above theorem we can conclude that the computation of the average defined in \eqref{eq:sumRBM} approximates the original one with an accuracy $O((1/M-1/N)^{1/2})$.

At the computational level, the RBM method, as it has been presented in \cite{AP,JLL,CPZ}, can be summarized as in Algorithm \ref{alg:RBM}.

\begin{algorithm}\label{alg:RBM}
    \caption{Random Batch Method}
    \begin{algorithmic}[1]
    \State Sample $N$ particles from the initial distribution $f(x,v,0)$
     \State Set $M\in \mathbb N$, $1<M<N$
 \For {each particle $i$}
 \State{ select a random batch $\mathcal S_M$}
\State update $(x_i,v_i)$ by solving a time discretization of \eqref{eq:system_RBM} 
\EndFor
\State Reconstruct the kinetic density using a kernel density estimator of \eqref{eq:fN}.
\end{algorithmic}
\end{algorithm}

\subsection{The reduced variance strategy}

In this section we want to define a suitable strategy to reduce the variance of the RBM method. To this end, we will develop a control variate approach which is based on the construction of simpler linear models which define a compatible asymptotic distribution with respect to the full model. We describe the method using a nonlocal PDEs of Vlasov-Fokker-Planck-type introduced in \eqref{eq:VFP}.  

We define a new drift term as follows
\begin{equation}
\label{eq:Plambda}
P^\lambda(x,y,v,w)(w-v) = P(x,y,v,w)(w-v)-\lambda (\tilde P(x,v)(w-v) - \tilde Q(f)(x,v,t)),
\end{equation}
where $\lambda \in \mathbb R$ and $\tilde Q(x,v,t)$ is a reduced complexity collision operator having the form
\[
\begin{split}
\tilde Q(f)(x,v,t)  &= \int_{\mathbb R^{2d}} \tilde P(x,v)(w-v)f(y,w,t)dy\,dw\\
&=\tilde P(x,v)(U-v),
\end{split}\]
being $\tilde P(x,v)\ge0$ a new interaction function. Hence, \eqref{eq:Plambda} may be rewritten as
\be
\begin{split}
P^\lambda(x,y,v,w)(w-v) &= P(x,y,v,w)(w-v)-\lambda (\tilde P(x,v)(w-v) - \tilde P(x,v)(U-v)) \\
&= P(x,y,v,w)(w-v)-\lambda \tilde P(x,v)(w-U).
\end{split}
\ee
We can notice that we have
\[
\E_f[P^\lambda(x,\cdot,v,\cdot)(\cdot-v)] = \E_f[P(x,\cdot,v,\cdot)(\cdot-v)]=Q(f)(x,v),
\]
so that the solution of our original problem can be obtained by the evaluation of the expectation of the new term $P^\lambda(x,\cdot,v,\cdot)(\cdot-v)$. In order to determine the optimal value of the scalar value $\lambda$ a common strategy relies on the minimization of the variance \cite{DP,DP2020,G}. In particular, the variance of the above expression is given by
\[
\V_f(P^\lambda(x,\cdot,v,\cdot)(\cdot-v))=\int_{\R^{2d}} (P^\lambda(x,y,v,w)(w-v)-Q(f)(x,v))^2 f(y,w)\,dw\,dy.
\]
From \eqref{eq:Plambda} we get
\[
\begin{split}
\V(P^\lambda(x,\cdot,v,\cdot)(\cdot-v)) =& \V(P(x,\cdot,v,\cdot)(\cdot-v)) + \lambda^2\V(\tilde P(x,v)(\cdot-U))\\
&-2\lambda\C(P(x,\cdot,v,\cdot)(\cdot-v),\tilde P(x,v)(\cdot-U)), 
\end{split}
\]
where
\[
\begin{split}
\C_f(P(x,\cdot,v,\cdot)(\cdot-v),\tilde P(x,v)(\cdot-U))=\int_{\R^{2d}} &(P(x,y,v,w)(w-v)-Q(f)(x,v))\\
&(P(x,v)(w-U)-\tilde Q(f)(x,v)) f(y,w)\,dw\,dy.
\end{split}
\]
We can minimize the variance by direct differentiation with respect to $\lambda$ and obtain the optimal value. The following result holds
\begin{theorem}
The quantity 
\be
\label{eq:lambdas}
\lambda^* = \frac{\C_f(P(x,\cdot,v,\cdot)(\cdot-v),\tilde P(x,v)(\cdot-U))}{\V_f(\tilde P(x,v)(\cdot-U))}.
\ee
minimizes the variance of $P^\lambda(x,y,v,w)(w-v)$ and is such that 
\[
\V(P^{\lambda^*}(x,y,v,w)(w-v)) = (1-\rho^2)\V(P(x,y,v,w)(w-v)),
\]
being $\rho\in [-1,1]$ the Pearson's correlation coefficient. 
\end{theorem}

Assume now to evaluate in \eqref{eq:Plambda} the exact expectation $\tilde P(x,v)(U-v)$ of the surrogate linear model using $N$ samples
\[
\tilde P(x_i,v_i)(U-v_i) \approx \tilde P(x_i,v_i) (U_N-v_i)
\]
where 
\[
U_N = \frac1{N} \sum_{j=1}^N v_j.
\]
We can improve the standard Monte Carlo random batch estimate based on $M$ samples \eqref{eq:system_RBM} using 
\begin{equation}
\label{eq:rvRBM}
\begin{split}
\frac1{M} \sum_{j\in S_M} P^{\lambda_*}(x_i,x_j,v_i,v_j)(v_j-v_i) =& \frac1{M} \sum_{j\in S_M} P(x_i,x_j,v_i,v_j)(v_j-v_i)\\
&-\lambda_*\left(\tilde P(x_i,v_i) (U_M-U_N)\right).
\end{split}
\end{equation}
Hence, we get the following system of SDEs for $(x_i,v_i)$
\begin{equation}
\begin{cases}
dx_i = &v_idt, \\
dv_i = & \left[\dfrac1{M} \displaystyle\sum_{j\in S_M} P(x_i,x_j,v_i,v_j)(v_j-v_i)\right.\\
&\left.-\lambda_*\left(\tilde P(x_i,v_i) (U_M-U_N)\right)\right]dt + \sqrt{2\sigma^2 }dW_i^t,
\end{cases}
\end{equation}
where as usual $\{W_i\}$, $i=1,\dots,N$ defines a set of independent Wiener processes. The reduced variance strategy can be summarised as in Algorithm \ref{alg:rvRBM}.

\begin{algorithm}
    \caption{Reduced Variance Random Batch Method (rvRBM)}
    \begin{algorithmic}[1]
    \State Sample $N$ particles from the initial distribution $f(x,v,0)$
     \State Set $M\in \mathbb N$, $1<M<N$
    \State Chose $\tilde P(x_i,v_i)$ that defines a surrogate model 
    \For{each particle $i$}
    \State{select a random batch $\mathcal S_M$}
    \State compute $\lambda^*$ as in \eqref{eq:lambdas}
\State update $(x_i,v_i)$ by solving a time discretization of \eqref{eq:rvRBM} 
\EndFor
\State Reconstruct the kinetic density using a kernel density estimator of \eqref{eq:fN}.
\end{algorithmic}
\label{alg:rvRBM}
\end{algorithm}

\begin{remark}
We may observe that $\C_f(P(x,y,v,w)(w-v),\tilde P(x,v)(w-U))$ and $\V_f(\tilde P(x,v)(w-U))$ defining $\lambda^*$ are not known are must be estimated starting from the sample of size $N$. In particular, the following unbiased estimators can be defined
\[
\begin{split}
&\mathbb E_{f^N}(\tilde P(x,v)(w-U)) = \dfrac{1}{N}\sum_{i=1}^N \tilde P(x,v)(w_i-U_N),\\
&\V_{f^N}(\tilde P(x,v)(w-U)) \\
&\qquad= \dfrac{1}{N-1}\sum_{i=1}^N \left(\tilde P(x,v)(w_i-U_N) - \mathbb E_{f^N}(\tilde P(x,v)(w-U_N))  \right)^2, \\
&\C_{f^N}(P(x,y,v,w)(w-v),\tilde P(x,v)(w-U_N)) \\
&\qquad= \dfrac{1}{N-1}\sum_{i=1}^N  \left(\tilde P(x,v)(w_i-U_N) - \mathbb E_{f^N}(\tilde P(x,v)(w_i-U_N))  \right) \\
&\qquad\quad \left(P(x,y_i,v,w_i)(w_i-v) - \mathbb E_{f^N}( P(x,y_i,v,w_i)(w_i-v))  \right), 
\end{split}\]
which lead to the estimated optimal $\lambda^*$ as 
\[
\lambda^*_N = \frac{\C_{f^N}(P(x,y,v,w)(w-v),\tilde P(x,v)(w-U_N))}{\V_{f^N}(\tilde P(x,v)(w-U_N))}.
\]
\end{remark}

\begin{remark}\label{rem:lambdak}
If we can consider a linear surrogate operator, i.e. $\tilde Q\left(\sum_{k=1}^K f_k\right) = \sum_{k=1}^K\tilde Q(f_k)$, where
\[
\tilde Q(f_k)(x,v,t) = \tilde P(x,v)(U_k-v), \qquad k = 1,2,
\]
and $U_k = \int_{\mathbb R^{2d}}vf_k(x,v,t)dvdx$ with $\sum_{k=1}^KU_k= U$, we can define 
\[
P^{\lambda_k}(x,y,v,w)(w-v) = P(x,y,v,w)(w-v) -\lambda_k\tilde P(x,v)(w-U_k),
\]
which leads to sequence of optimal values
\begin{equation}
\label{eq:lambdak}
\lambda^*_k = \frac{\C_f(P(x,\cdot,v,\cdot)(\cdot-v),\tilde P(x,v)(\cdot-U_k))}{\V_f(\tilde P(x,v)(\cdot-U_k))}.
\end{equation}
\end{remark}

\section{Numerical tests}
In this section we present several numerical examples to show the performances of the introduced reduced variance RBM. In the following we will recall both deterministic and stochastic models.  In all the subsequent tests, we reconstruct the particle densities through a regularization of the empirical distribution $f^N(x,v,t) = \dfrac{1}{N} \sum_{i=1}^N \delta(x-x_i)\otimes \delta(v-v_i)$ based on a kernel density estimator obtained through the evaluation of $\left\langle \varphi,f^N \right \rangle$, being $\varphi(x_i,v_i)$ the normal density centered in $(x_i,v_i)$ and having fixed variance $\Sigma^2 = 10^{-5}$. 

In the following numerical tests we will focus on the definition of suitable RBM and rvRBM methods pointing the reader to the nonlocal PDE models introduced in Section \ref{sect:2}. 

\subsection{Test 1: alignment models}

\subsubsection{Test 1a: bounded confidence model}
We consider the bounded confidence opinion model introduced in Section \ref{sect:2.3} in the case $D\equiv 0$. We recall that it is possible to prove that \eqref{eq:BCMF} corresponds to the mean-field limit of the following particle system
\begin{equation}
\label{eq:BCparticles}
\dfrac{d}{dt}  v_i(t) = \dfrac{1}{N} \sum_{j=1}^NP(| v_i- v_j|) ( v_j- v_i), \qquad i = 1,\dots,N,
\end{equation}
where $P(|v_i-v_j|) = \chi(|v_i-v_j|\le \delta)$, $\delta >0$, see e.g. \cite{BL,PTTZ} and the references therein. Hence, the RBM method corresponds to the following reduced complexity model 
\begin{equation}
\label{eq:RBM_BC}
\dfrac{d}{dt}  v_{i}(t) = \dfrac{1}{M} \sum_{j\in \mathcal S_M} P(| v_{i}- v_{j}|) ( v_{j}- v_{i}), \qquad i = 1,\dots,N,
\end{equation}
being $\mathcal S_M$ a uniform subsample of size $M>1$ obtained from the $\{v_i\}_{i=1}^N$. A rvRBM method can be therefore obtained by considering in \eqref{eq:rvRBM} the interaction function $\tilde P(v_i)$. In particular, we will consider the cases $\tilde P(v_i)\equiv 1$ (denoted by case 1) and $\tilde P(v_i) = 1-v_i^2$ (denoted by case 2). 
Therefore the rvRBM method corresponds in considering the evolution of the following particle system
\begin{equation}
\label{eq:rvRBM_BC}
\dfrac{d}{dt}v_{i}(t) = \dfrac{1}{M} \sum_{j\in \mathcal{S}_M}P(|v_i-v_j|)(v_j-v_i) - \lambda_*\tilde P(v_i)(u_M-u_N),
\end{equation}
where 
\[
u_M = \dfrac{1}{M} \sum_{j\in \mathcal{S}_M} v_j, \qquad u_N = \dfrac{1}{N} \sum_{j=1}^N v_j,
\]
and $\lambda^* = \lambda^*(t)$ is computed as in \eqref{eq:lambdas}. 
We remark that, since $u_N$ is a conserved quantity of the model it is sufficient to compute $u_N$ from the initial sample, whereas $u_M$ is affected from the error induced by the uniform subsampling. 

In Figure \ref{fig:BC} we report the evolution of the particle density at time $t = 1$ and $t=5$ and by considering $M=5$ or $M=10$. As initial distribution we considered a sample of the uniform distribution of the interval $[-1,1]$, i.e.
\begin{equation}
\label{eq:f0BC}
f(v,0) =
\begin{cases}
\frac{1}{2} & v \in [-1,1], \\
0		&\textrm{elsewhere}. 
\end{cases}
\end{equation}
We compare the evolution of the particle density obtained with a standard RBM together with the one obtained with rvRBM, from which we may notice  good agreement in describing the transient behaviour of the distribution both in the case $\delta = 1$ (top row) and in the case $\delta = 0.5$ (bottom row). In this latter case the emerging equilibrium density displays clusterization towards the values $\pm \frac{1}{2}$. 

\begin{figure}
\centering
\subfigure[$t=1,M=5$]{
\includegraphics[scale = 0.17]{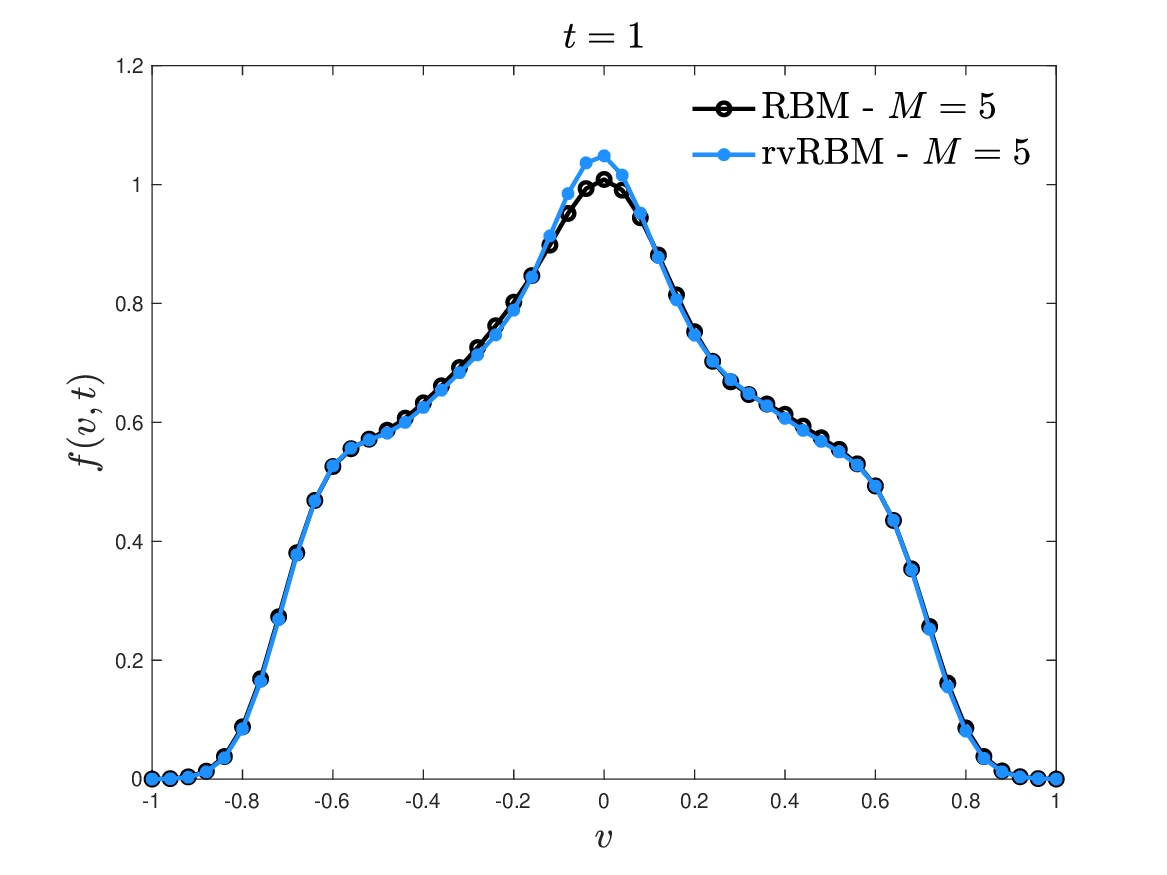}}
\subfigure[$t=1,M=10$]{
\includegraphics[scale = 0.17]{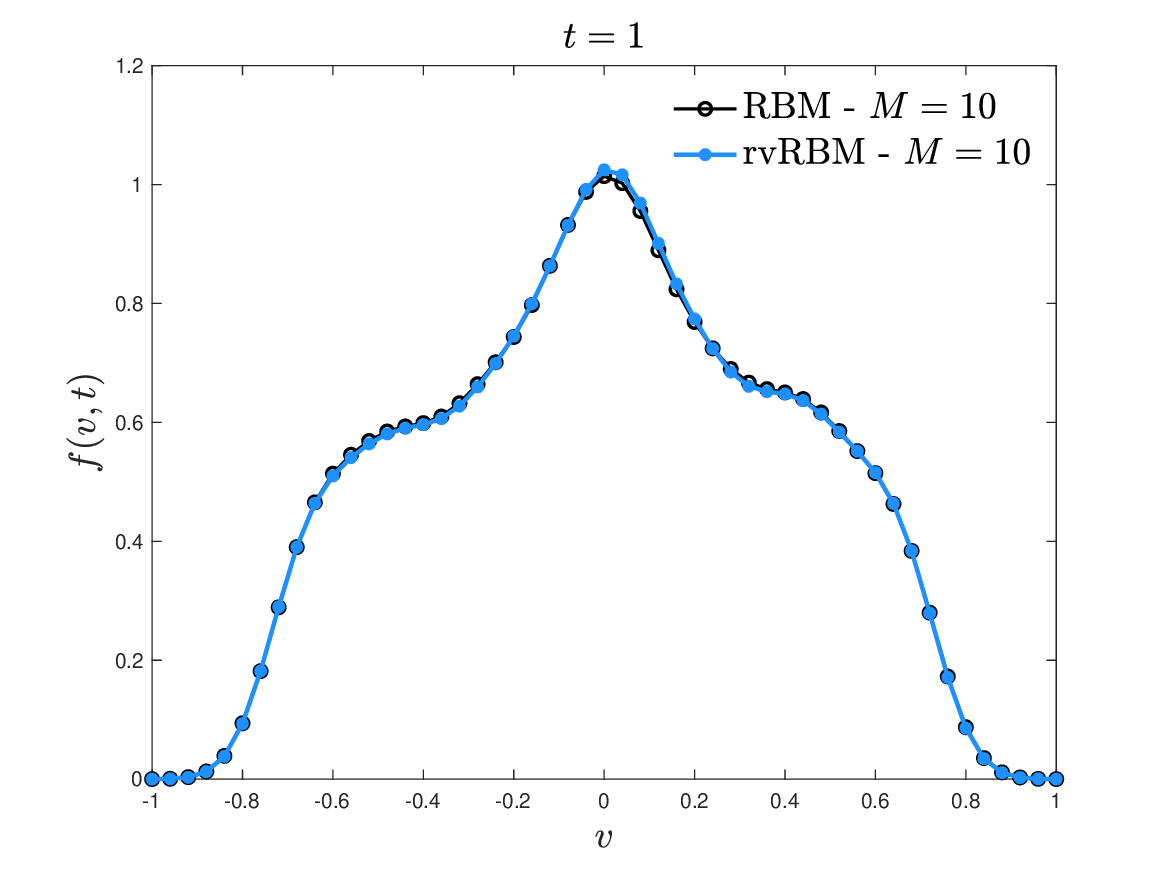} }
\subfigure[$t=5,M=5$]{
\includegraphics[scale = 0.17]{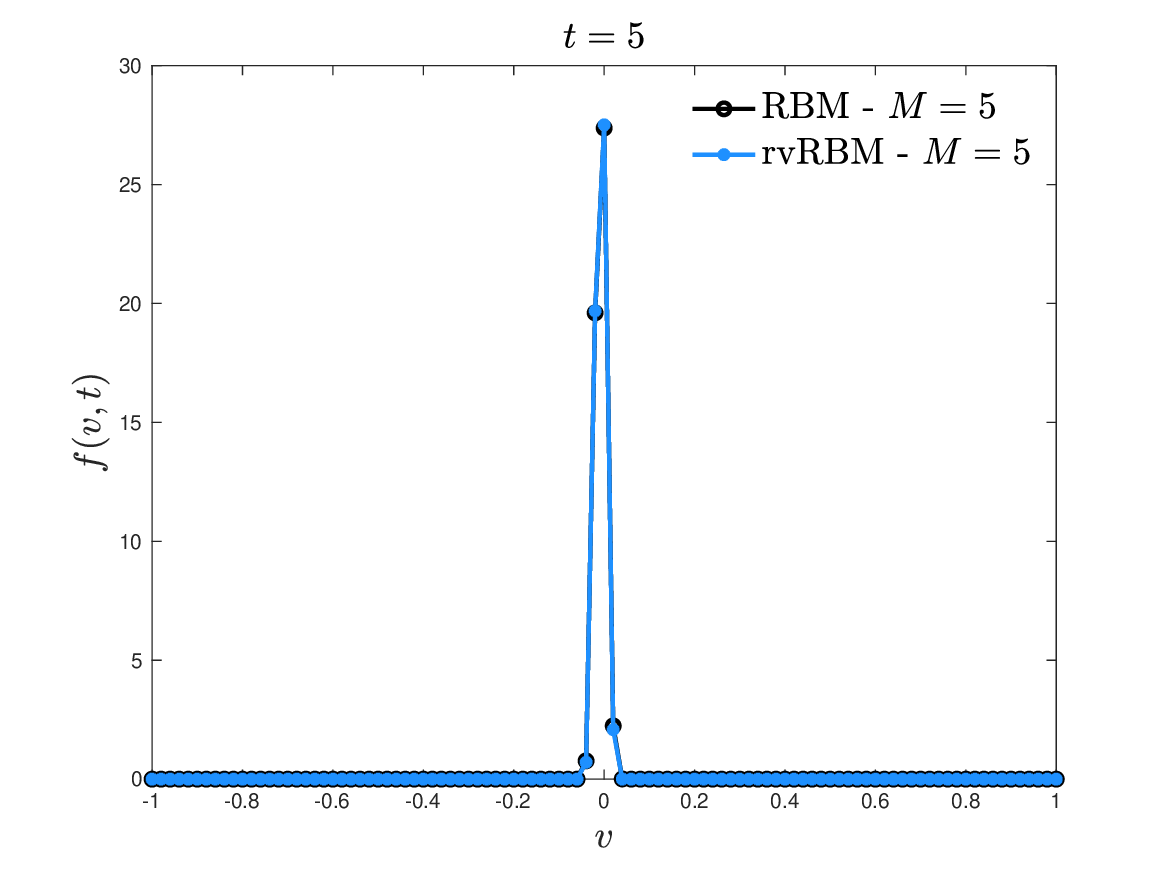}}
\subfigure[$t=5,M=10$]{
\includegraphics[scale = 0.17]{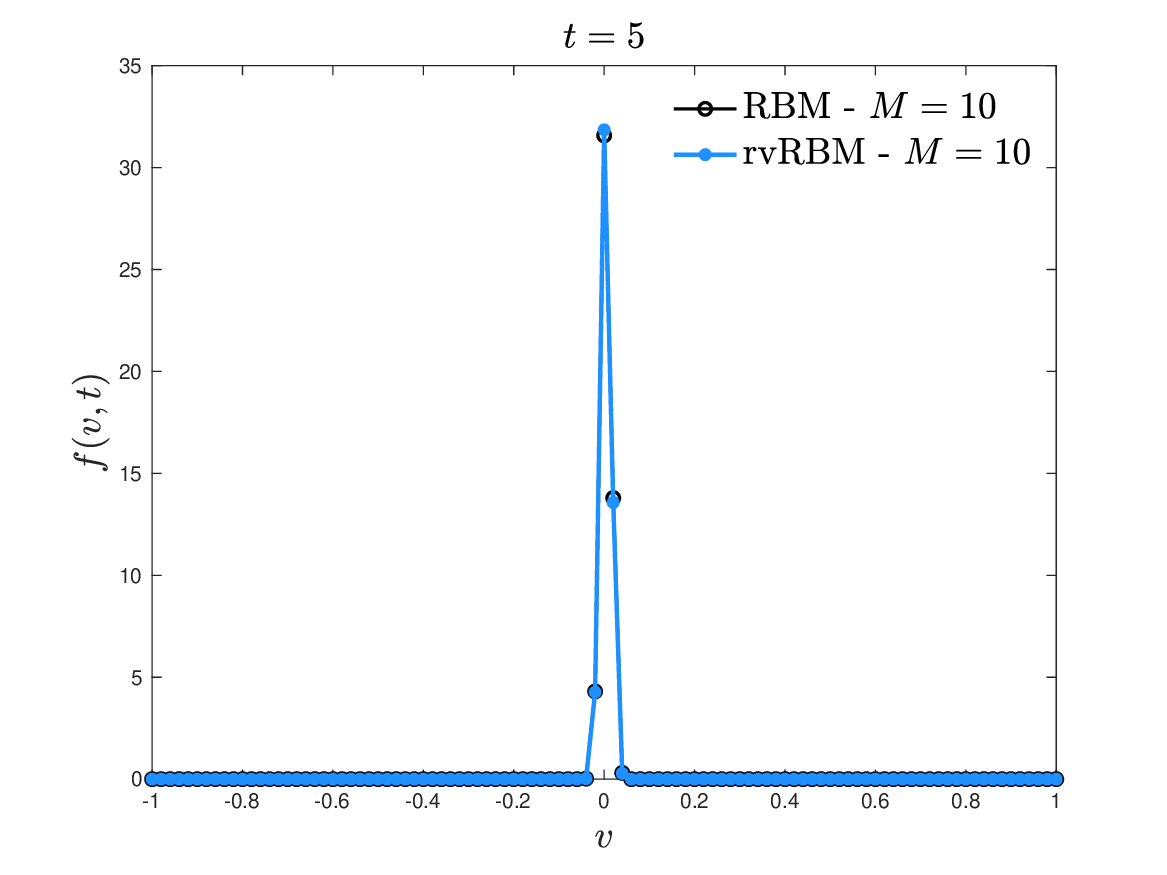}}\\
\subfigure[$t=1,M=5$]{
\includegraphics[scale = 0.17]{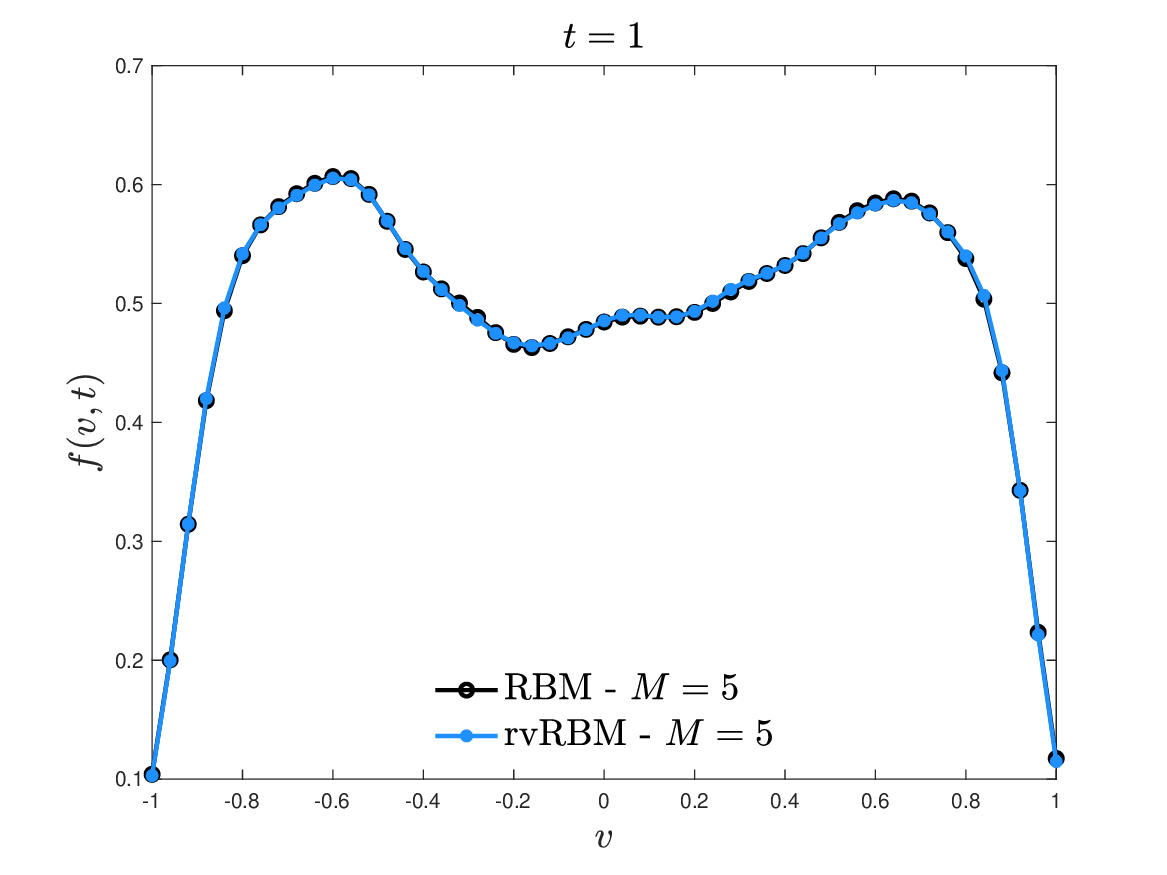}}
\subfigure[$t=1,M=10$]{
\includegraphics[scale = 0.17]{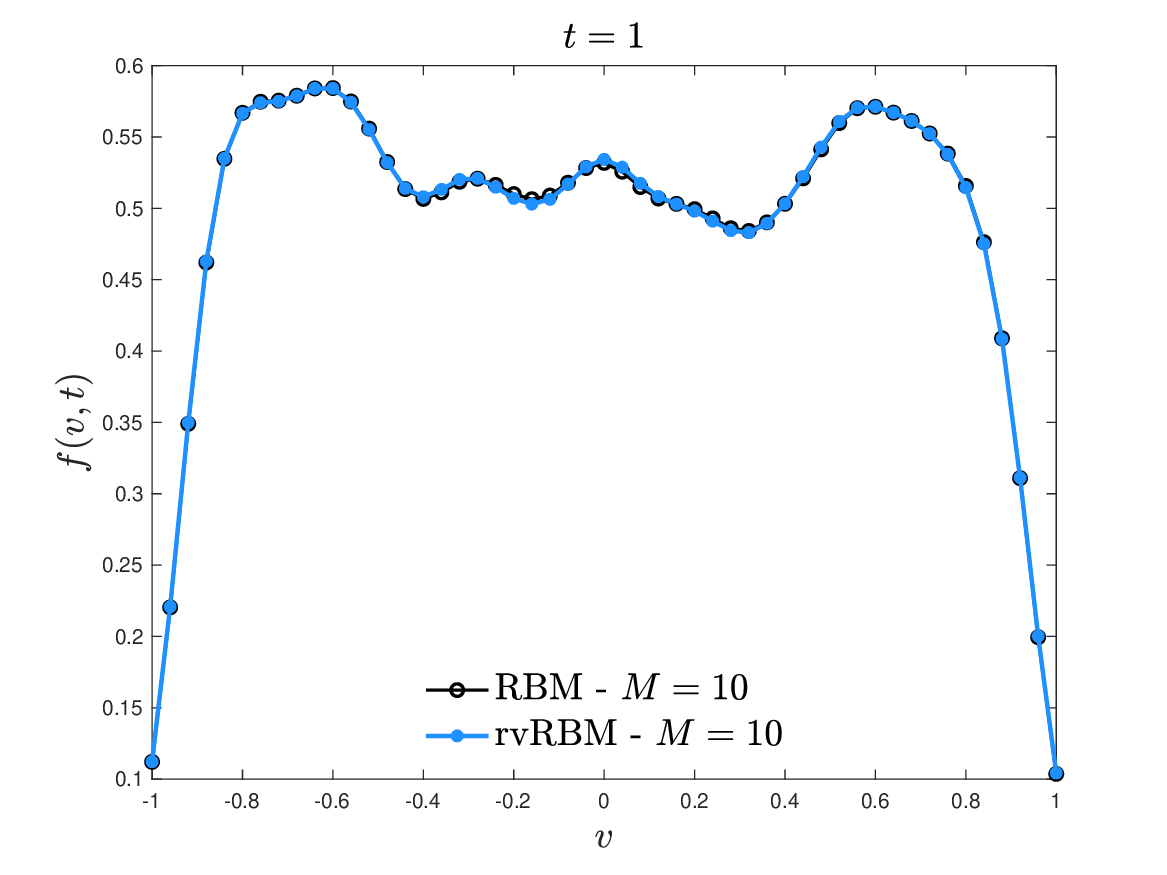} }
\subfigure[$t=5,M=5$]{
\includegraphics[scale = 0.17]{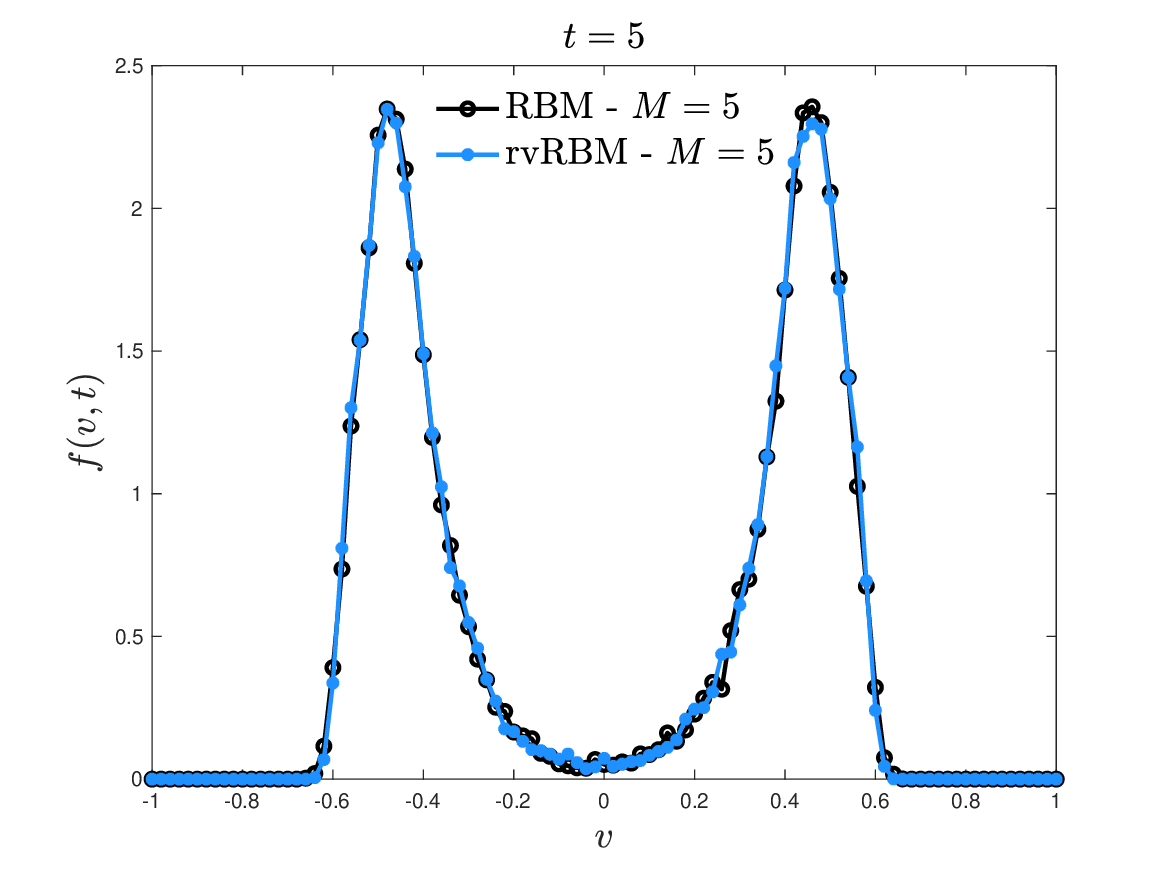}}
\subfigure[$t=5,M=10$]{
\includegraphics[scale = 0.17]{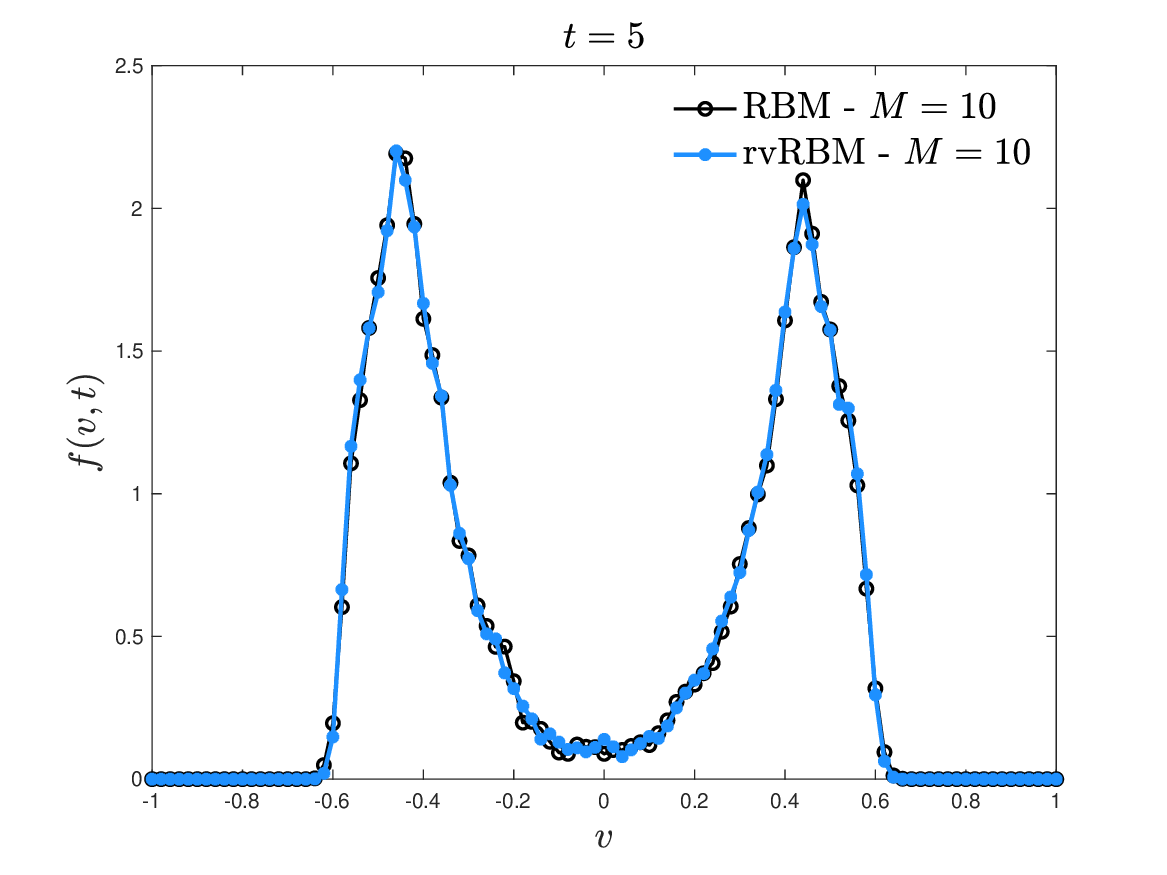}}
\caption{\textbf{Test 1a}. Evolution of the densities for the deterministic BC model in \eqref{eq:BCMF} by means of a particle approach defined in \eqref{eq:BCparticles} with interaction function with threshold $\delta = 1$ (top row), $\delta = 0.5$ (bottom row). We considered $N = 10^5$ particles with initial distribution defined in \eqref{eq:f0BC} and we compare the evolution of the approximated densities obtained with RBM and rvRBM with a subset of interacting particles of size $M=5$ or $M=10$. The surrogate interaction function is here considered $\tilde P(v_i) \equiv 1$.  }
\label{fig:BC}
\end{figure}

We consider the following absolute error for the mean
\begin{equation}
\label{eq:error}
\textrm{Error} = \left|\dfrac{1}{N} \sum_{i=1}^N v_{i} - m\right|,
\end{equation}
where $\left\{v_{i} \right\}_{i=1}^N$ is obtained either from RBM, as in \eqref{eq:RBM_BC}, or from rvRBM, as in \eqref{eq:rvRBM_BC}. 
In Figure \ref{fig:error_BCdeterministic} we show the evolution in time of the introduced absolute error in the case $\delta = 1$ (top row) and $\delta = 0.5$ (bottom row) by both the standard RBM and the introduced rvRBM method. We will refer to case 1 if $\tilde P(v_i)\equiv 1$ whereas with case 2 we will refer to $\tilde P(v_i) = 1-v_i^2$. We may observe how in the case $\delta = 1$ the method rvRBM leads to a significant advantage since the emerging distribution function is highly correlated with the one characterising the reduced variance method. Indeed, as shown in Figure \ref{fig:lambda_BCdet} (left plot) the optimal $\lambda^*$ is such that 
\[
\lambda^* \to 1, \qquad t\to +\infty. 
\]
On the other hand, for the case $\delta = 0.5$ the error produced by the rvRBM strategy is comparable with the one of RBM. Indeed, the computed $\lambda^*\to 0$, since the emerging distribution has no correlation with the one of the considered for the reduced variance method. 
\begin{figure}
\centering
\includegraphics[scale = 0.35]{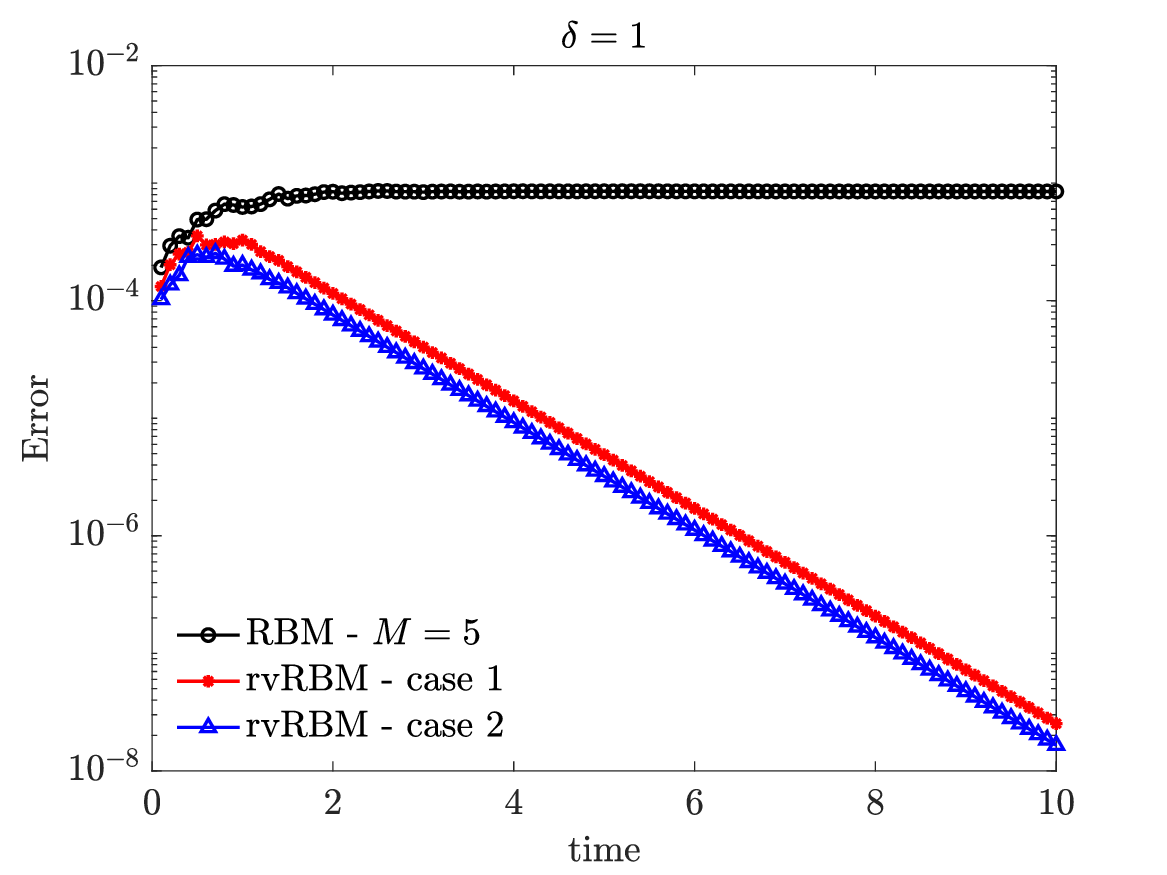}
\includegraphics[scale = 0.35]{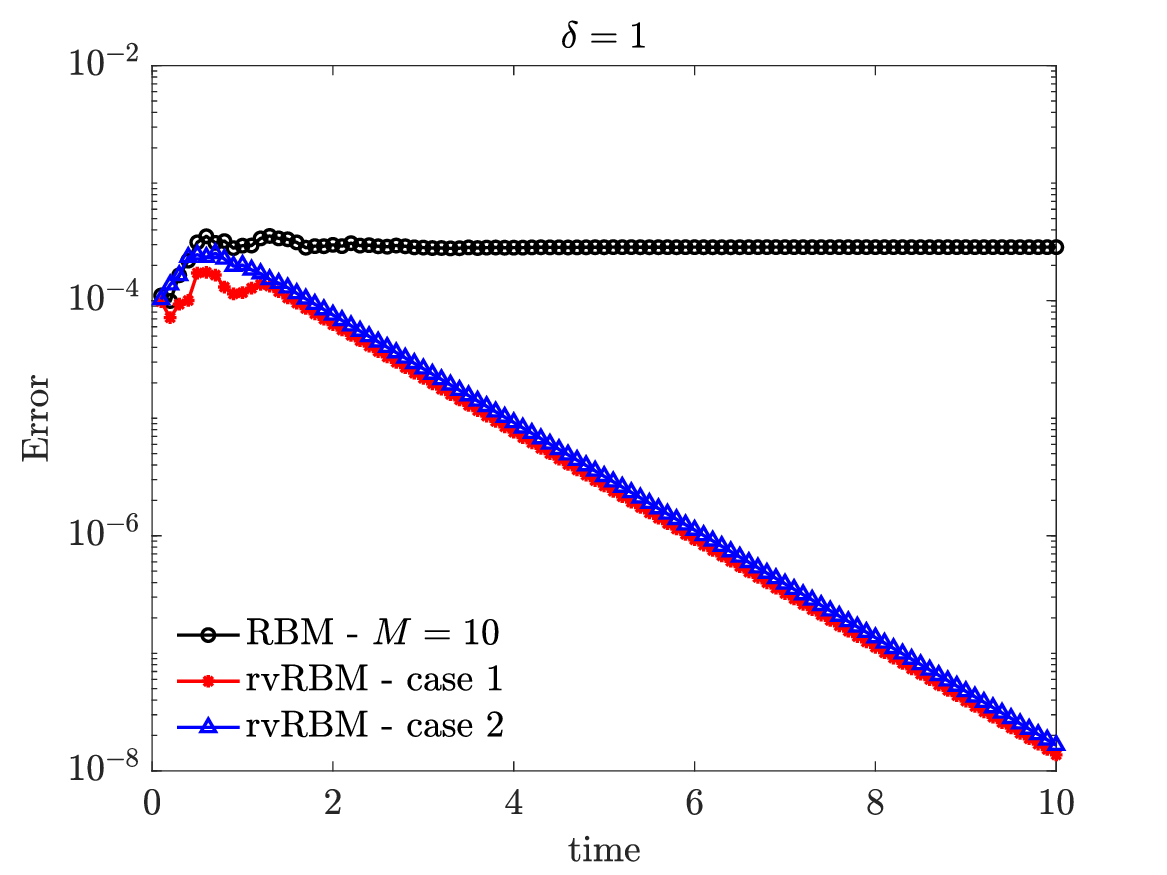} \\
\includegraphics[scale = 0.35]{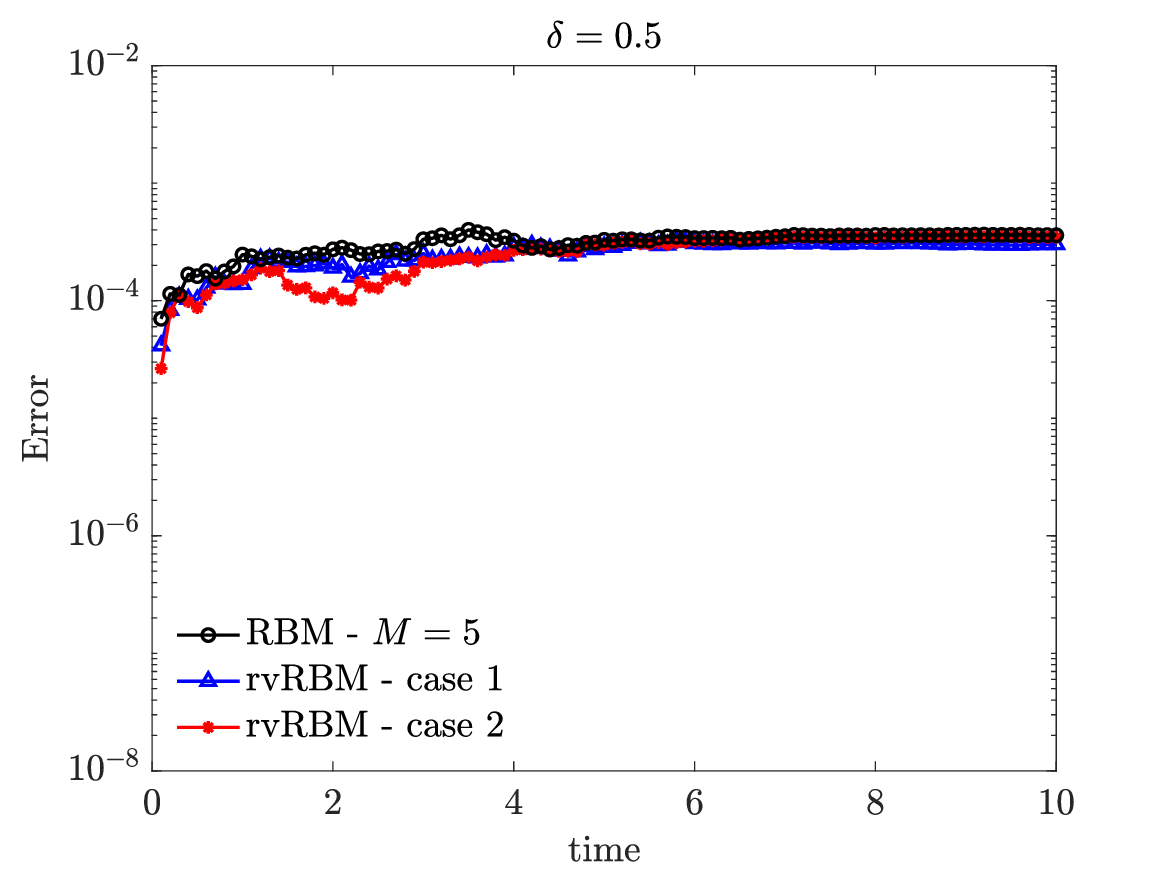}
\includegraphics[scale = 0.35]{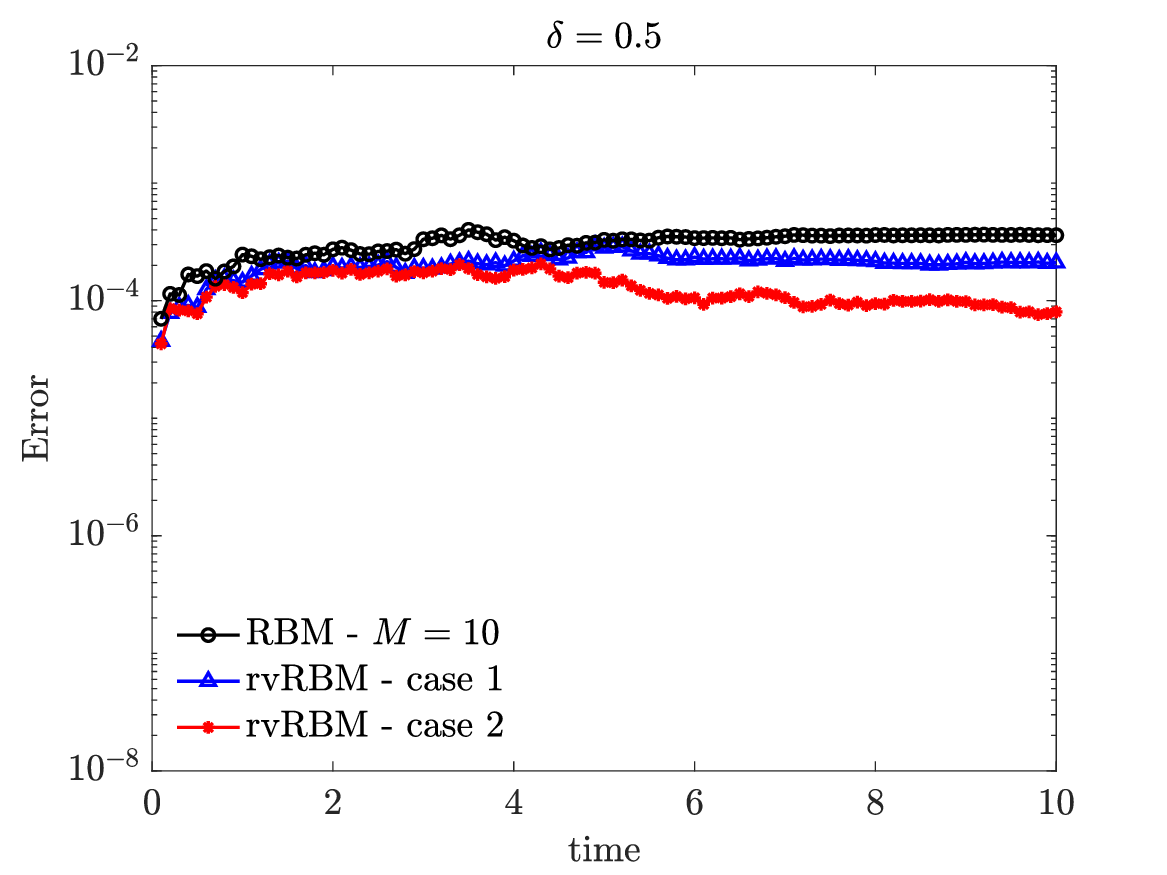}
\caption{\textbf{Test 1a.} Evolution of the absolute error defined in \eqref{eq:error} for the bounded confidence model solved through RBM and rvRBM with $M = 5$ (left column) or $M=10$ (right column). We considered $\delta=1$ (top row) and $\delta = 0.5$ (bottom row). }
\label{fig:error_BCdeterministic}
\end{figure}

\begin{figure}
\centering
\includegraphics[scale = 0.35]{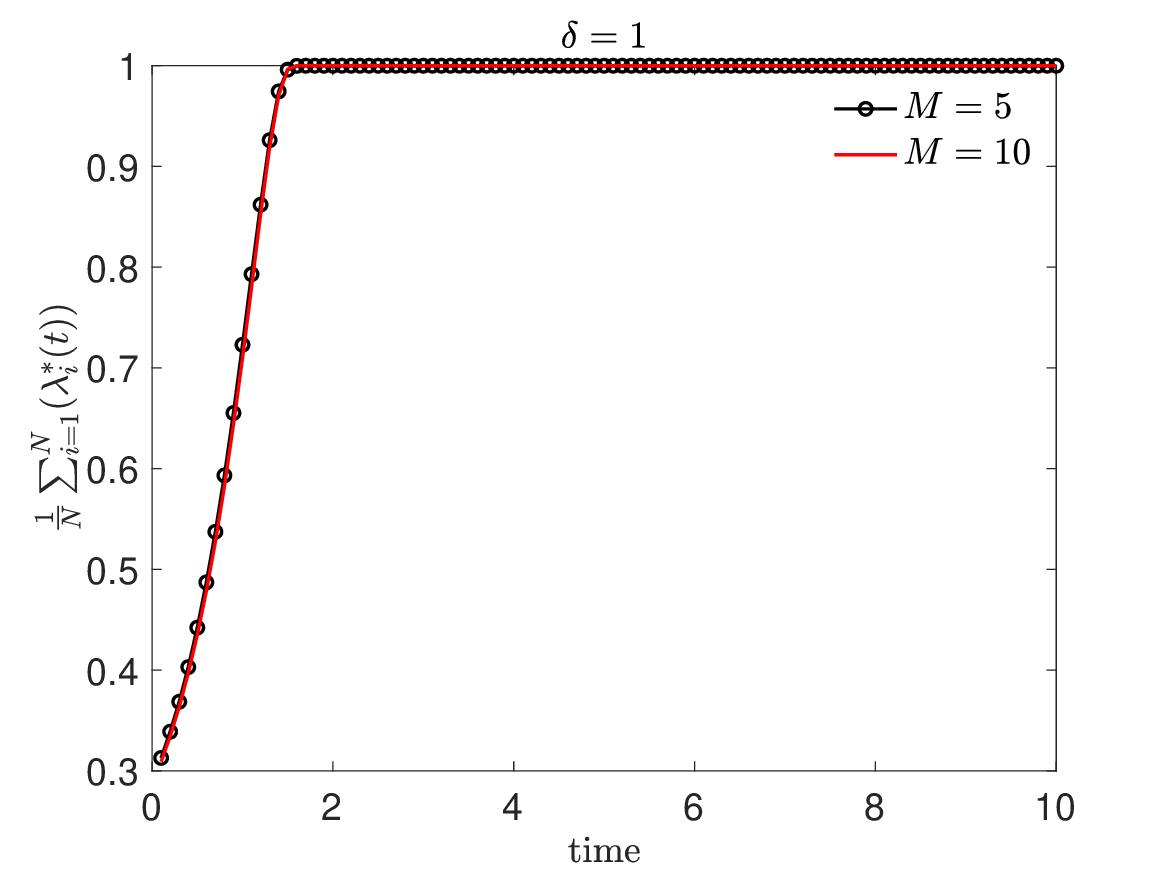}
\includegraphics[scale = 0.35]{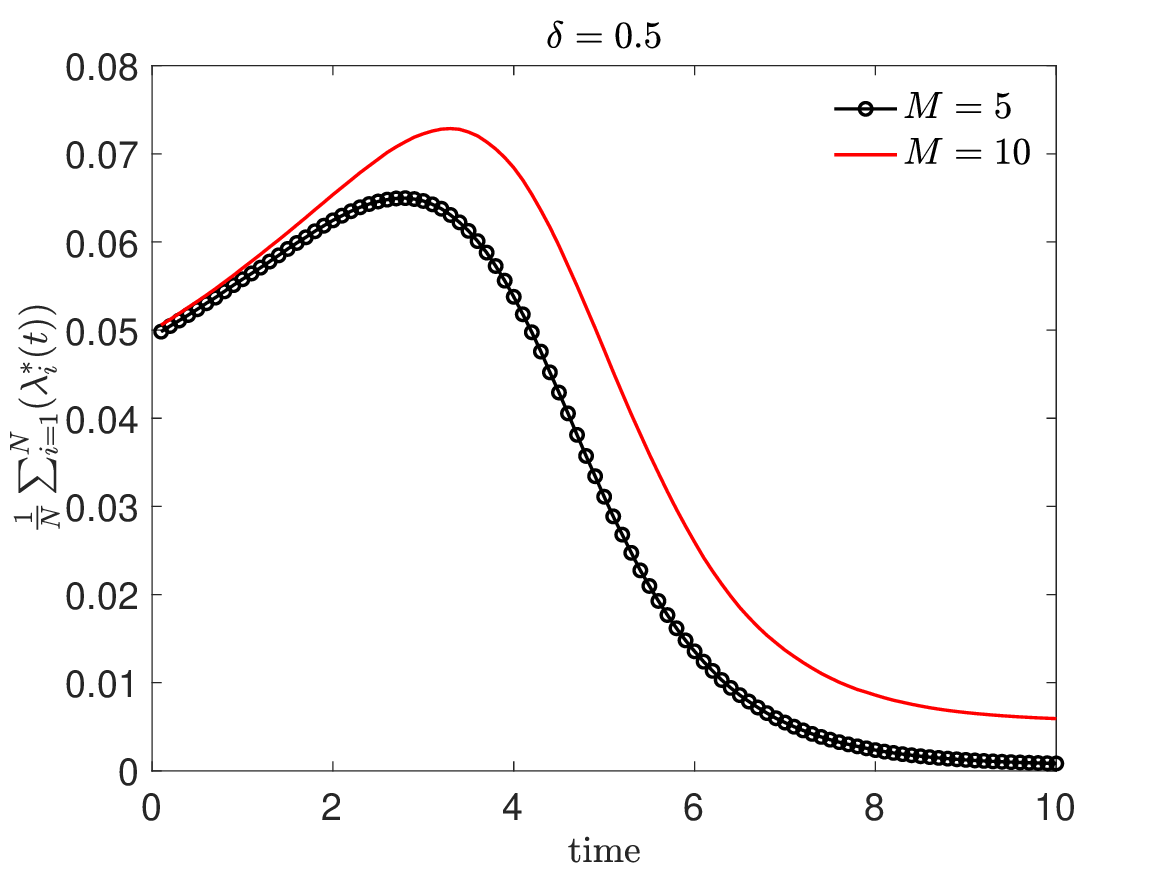}
\caption{\textbf{Test 1a.} Evolution of $\lambda^*$ for the model \eqref{eq:BCMF} approximated through a rvRBM strategy with $M=5,10$ and $\delta = 1$ (left), $\delta = 0.5$ (right). }
\label{fig:lambda_BCdet}
\end{figure}

In Figure \ref{fig:error_costN} we depict the evolution of the error produced by fixed batch $M= 5$ (left) or $M = 10$ (right) and a variable $N \in \{100,\dots,10^4\}$. The error \eqref{eq:error} is here computed at time $T = 5$. Since the cost of the reduced variance approach is compatible to the one of a RBM method $O(MN)$, we may observe how, in the reduced variance approach, we essentially gain accuracy at the same cost of RBM. 

\begin{figure}
\centering
\includegraphics[scale = 0.35]{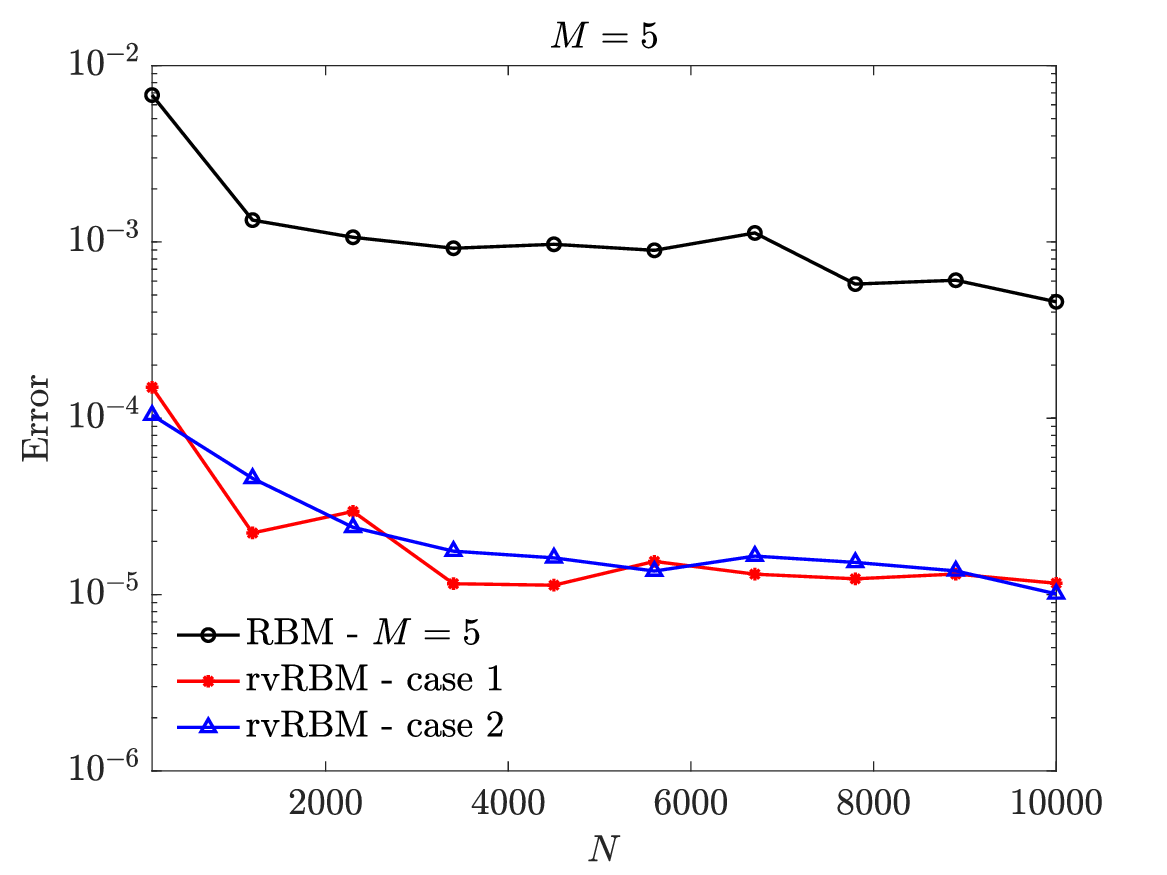}
\includegraphics[scale = 0.35]{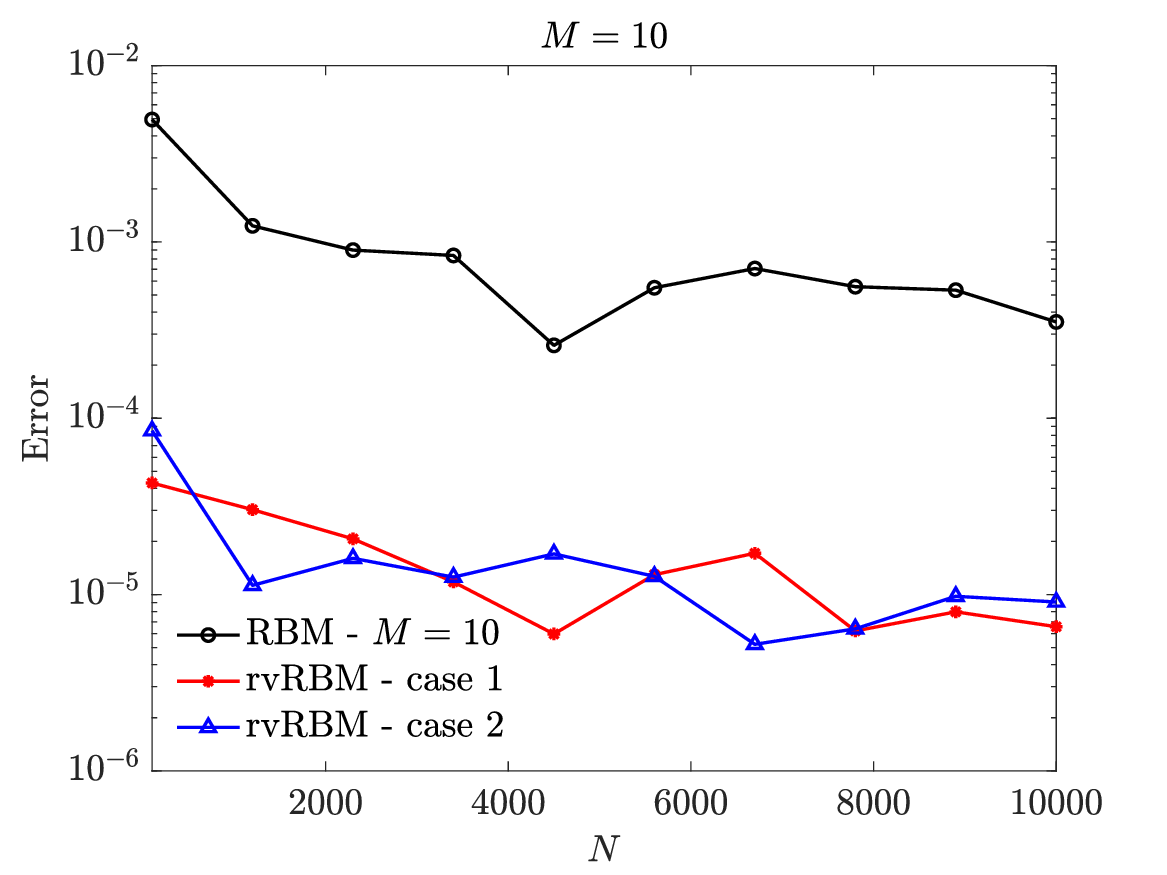}
\caption{\textbf{Test 1a.}  Comparison of the errors produced by RBM and rvRBM with surrogate models obtained with $\tilde P\equiv 1$ (case 1) or $\tilde P(v_i) = 1-v_i^2$ (case 2). We considered  a fixed batch size $M = 5$ (left) and $M = 10$ (right) and variable size of the sample $N \in \{10,\dots,10^4\}$. The error is computed from the bounded confidence model \eqref{eq:BCparticles} with $\delta = 1$ at time $T = 5$.   }
\label{fig:error_costN}
\end{figure}

\subsubsection{Test 1b: bounded confidence model with isolated clusters}
The loss of accuracy of the rvRBM method in the case $\delta = 0.5$ is essentially due to the uncorrelated behaviour of the surrogate model with respect to the full one. Indeed, whereas the considered surrogate models would lead to the formation of global consensus, the bounded confidence model leads to the formation of local stable clusters \cite{BL}. A possible way to overcome this issue has been addressed in Remark \ref{rem:lambdak}. We consider a sample of the initial distribution
\begin{equation}
\label{eq:f0_2c}
f(v,0) = 
\begin{cases}
1 & v \in [\frac{1}{4},\frac{3}{4}] \\
1 & v \in [-\frac{3}{4},-\frac{1}{4}]. 
\end{cases}
\end{equation}
Now, for $\delta = 0.5$ no interactions are expected from the two initial clusters centered in $\pm\frac{1}{2}$. We computed the optimal $\lambda^*_k$, $k=1,2$, exploiting two surrogate models, the first is $\tilde P(v_i)\equiv 1$ (case 1) whereas for the second we considered $\tilde P(v_i) = (v_i-\frac{1}{2})(v_i+\frac{1}{2})$ (case 2).
In Figure \ref{fig:2clusters} we report the evolution of the rvRBM method where we defined a sequence of optimal $\lambda^*_k$ as in \eqref{eq:lambdak} with $k=1,2$ such that $U_k = \pm \frac{1}{2}$. We may observe that the evolution in time of the obtained rvRBM method is consistent with the one defined by the standard RBM method. Furthermore, from the evolution of the absolute error \eqref{eq:error} we can observe that the method is capable to achieve much higher accuracy than RBM.  
\begin{figure}
\centering
\includegraphics[scale = 0.35]{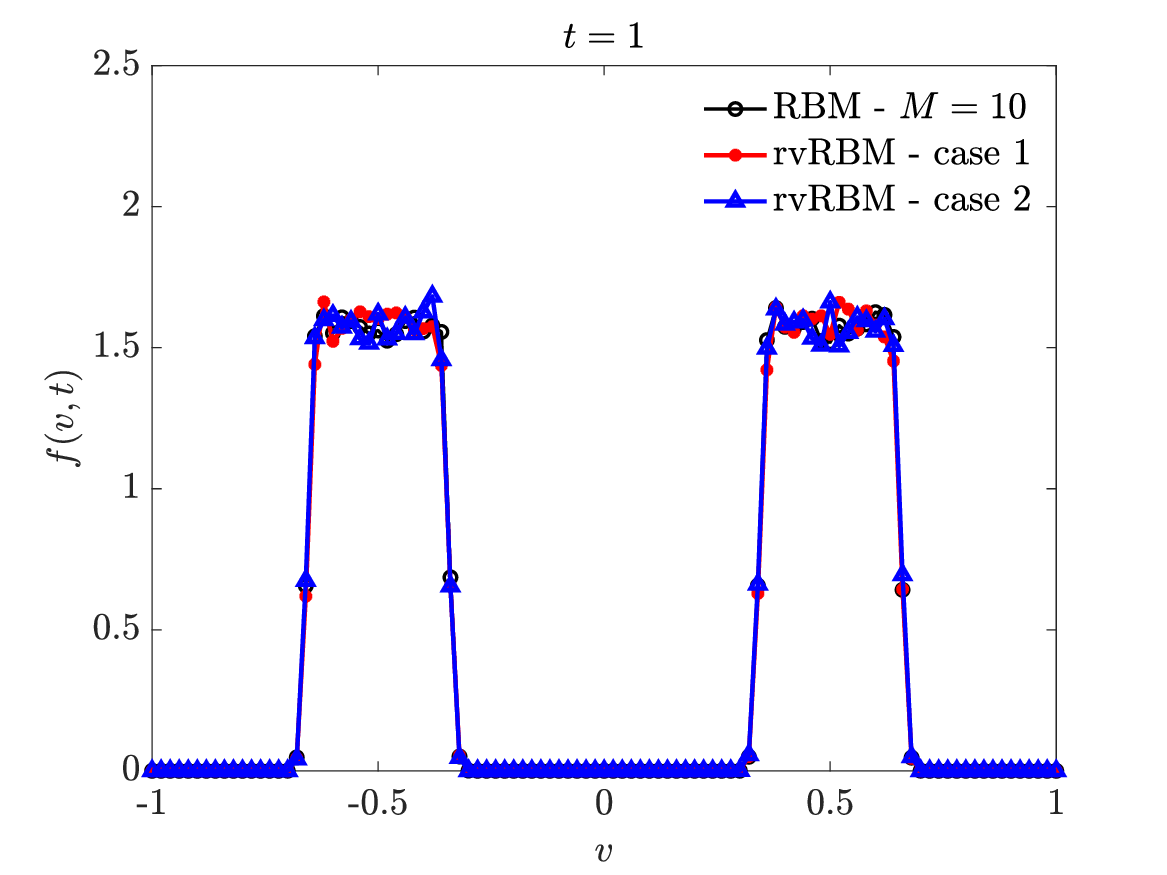}
\includegraphics[scale = 0.35]{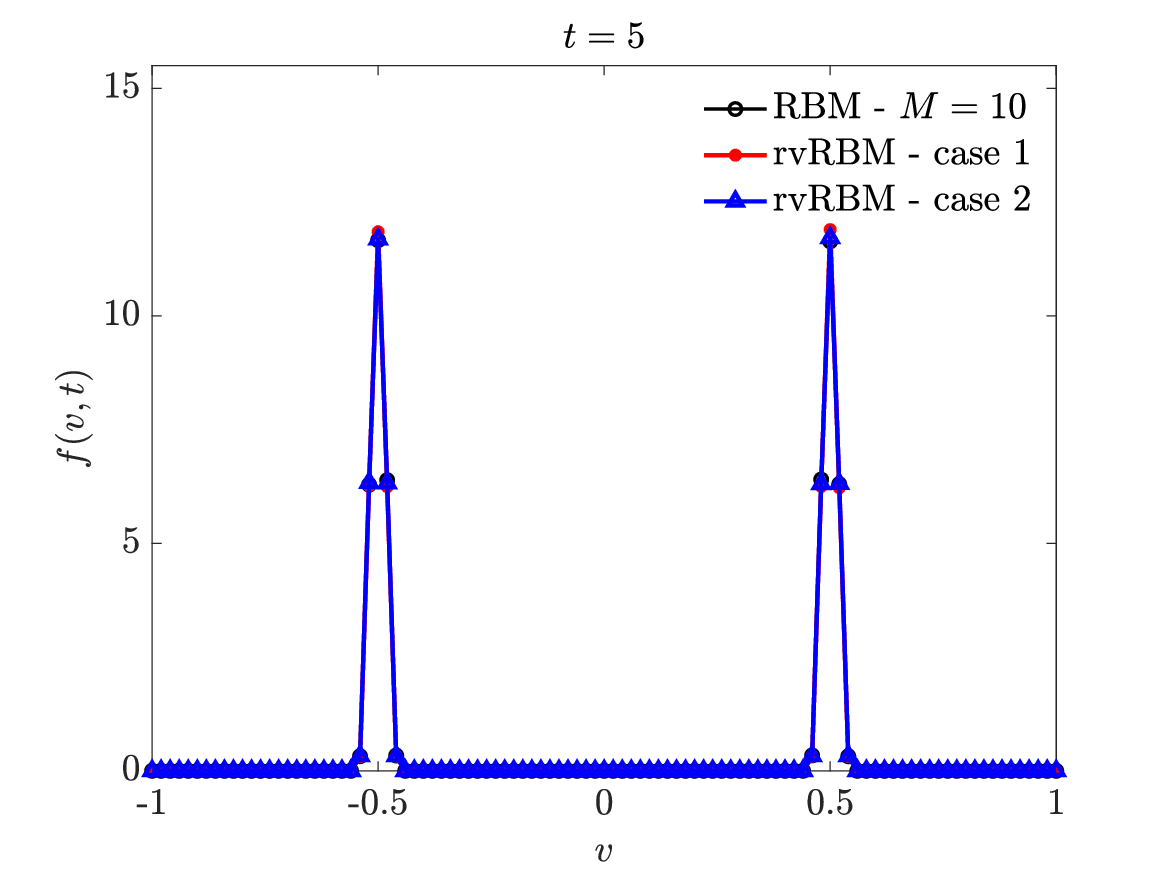}\\
\includegraphics[scale = 0.35]{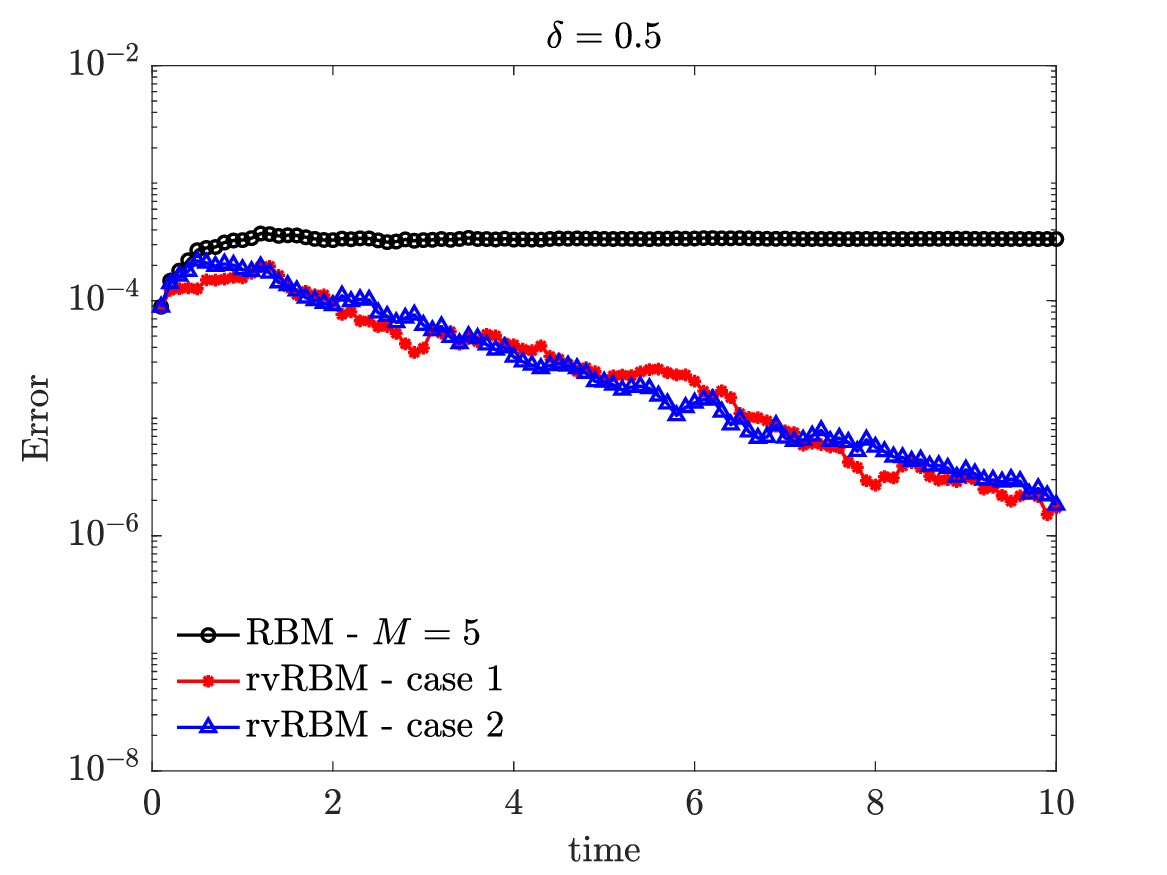}
\includegraphics[scale = 0.35]{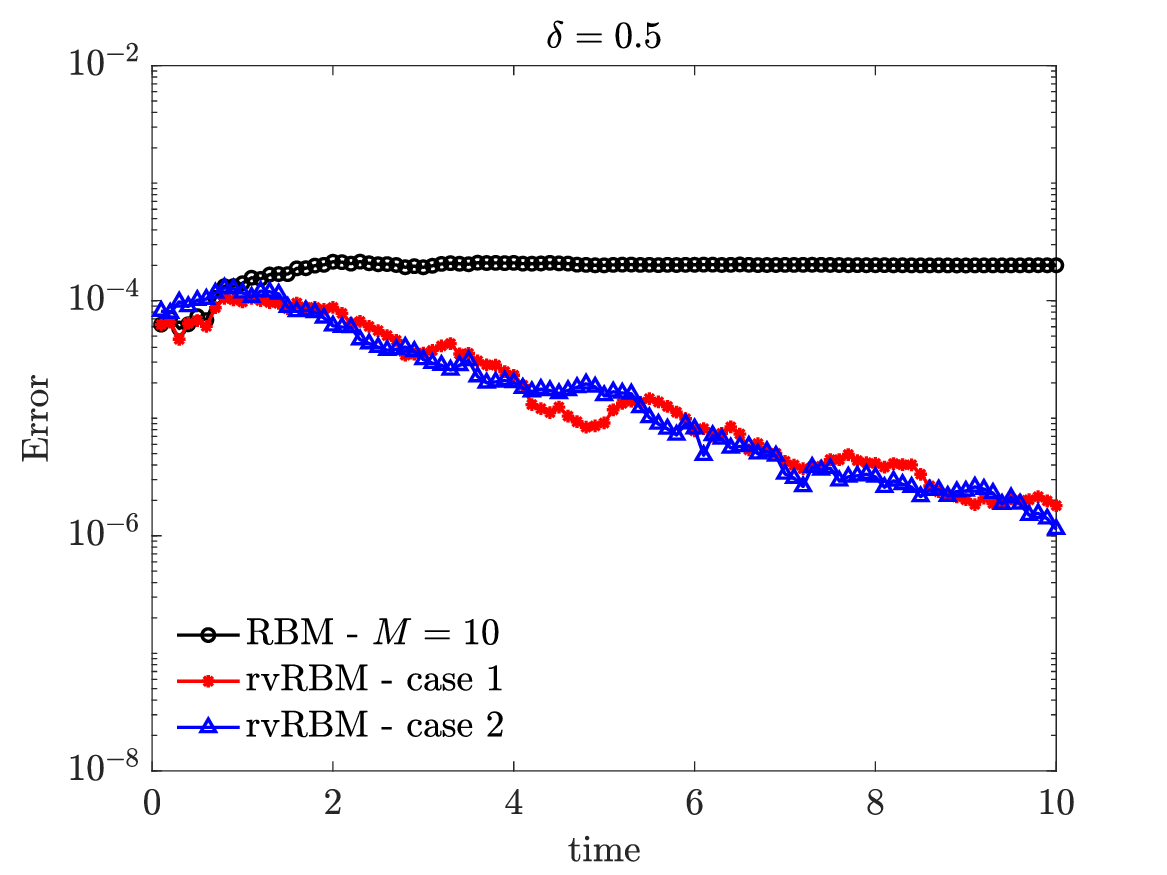}
\caption{\textbf{Test 1b.} Top row: comparison of the reconstructed distributions at time $t = 1$ and $t = 5$ for the bounded confidence model with $\delta = 0.5$ using RBM or rvRBM with $\tilde P(v_i)\equiv 1$ (case 1) or $\tilde P(v_i) = (v_i-\frac{1}{2})(v_i+\frac{1}{2})$. The batch size is $M = 10$ and the initial sample has been obtained from \eqref{eq:f0_2c}. Bottom row: evolution of the absolute errors \eqref{eq:error} for the bounded confidence model solved through RBM and rvRBM with $M = 5$ (left) or $M = 10$ (right). }
\label{fig:2clusters}
\end{figure} 

\subsection{Test 2: stochastic opinion formation model}
We now consider a stochastic bounded confidence model defined in Section \ref{sect:2.3} in \eqref{eq:BCstoch}, where the interaction function is defined as in \eqref{eq:BCinter} and that is obtained from the meanfield limit of the following system of stochastic differential equations
\begin{equation}
\label{eq:BCstochmicro}
dv_i(t) = \dfrac{1}{N} \sum_{j=1}^N P(|v_i-v_j|)(v_j-v_i)dt + \sqrt{2\sigma^2 D^2(v_i)} dW_i^t,
\end{equation}
with $D^2(v_i) = 1-v_i(t)$ and $\{W_i\}_{i=1}^N$ a set of independent Wiener processes. In the following, we fixed  the parameter $\sigma^2 = 10^{-1}$. Therefore, the RBM approximation is obtained by considering 
\[
dv_i(t) = \dfrac{1}{M} \sum_{j\in \mathcal S_M} P(|v_i-v_j|)(v_j-v_i)dt + \sqrt{2\sigma^2 D^2(v_i)} dW_i^t,
\]
where $\mathcal S_M$ is a uniform subsample of size $M>1$ from the $\{v_i\}_{i=1}^N$. Similarly, rvRBM strategy is obtained as 
\[
dv_i(t) = \left[\dfrac{1}{M} \sum_{j\in \mathcal S_M} P(|v_i-v_j|)(v_j-v_i) - \lambda^* \tilde P(v_i)(U_N-U_M)\right]dt + \sqrt{2\sigma^2 D^2(v_i)} dW_i^t.
\]
In the following, the methods defining RBM and rvRBM share the same set of Wiener processes of the complete model \eqref{eq:BCstochmicro}. The set of SDEs is solved through the Euler-Maruyama scheme. 

In Figure \ref{fig:BC_stoch} we report the evolution of the particle density at time $t = 1$ and $t = 5$ in the stochastic setting and obtained with $M = 5,10$. As initial distribution we considered the one introduced in \eqref{eq:f0BC}. In each plot, we compare the evolution of the approximated density obtained with $N = 10^5$ particles and either RBM or rvRBM algorithms. In Figure \ref{fig:BC_stoch} we considered the case $\delta = 1$. 
\begin{figure}
\centering
\subfigure[$t=1,M=5$]{
\includegraphics[scale = 0.17]{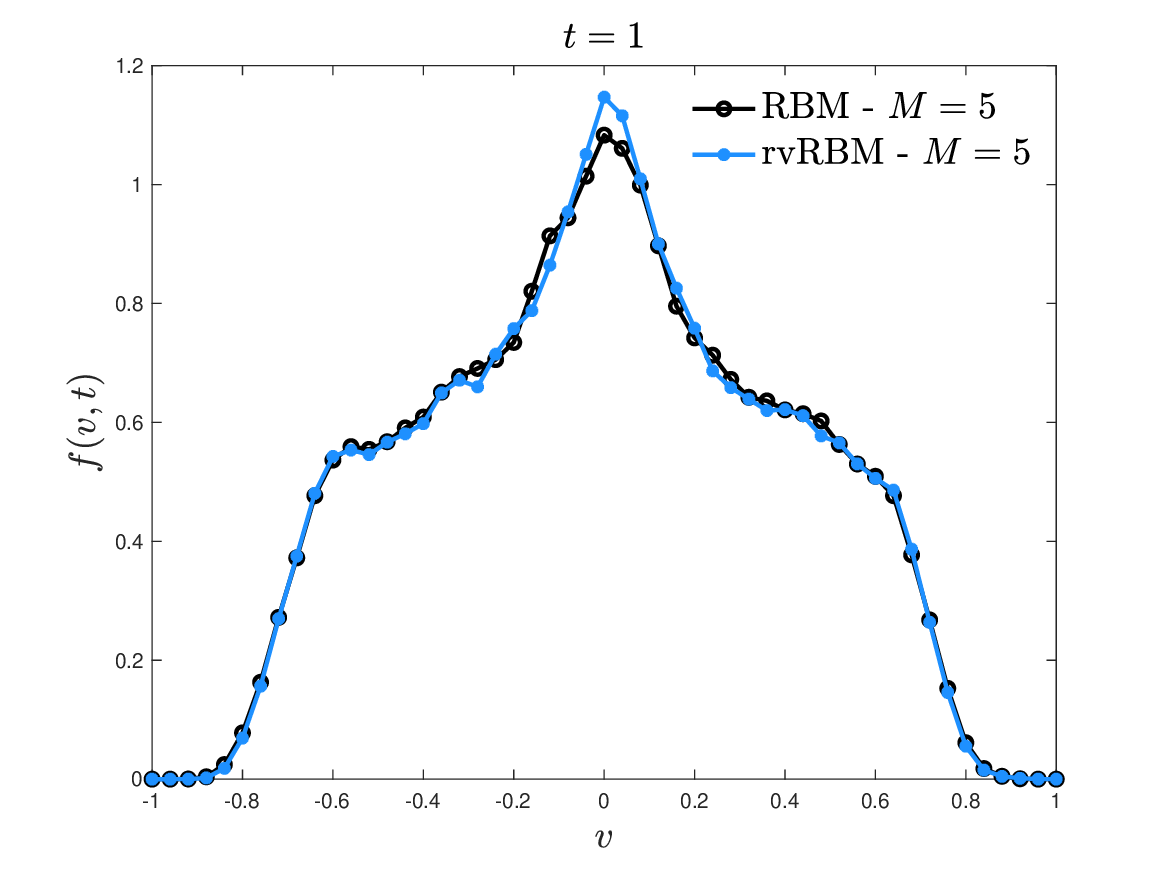}}
\subfigure[$t=1,M=10$]{
\includegraphics[scale = 0.17]{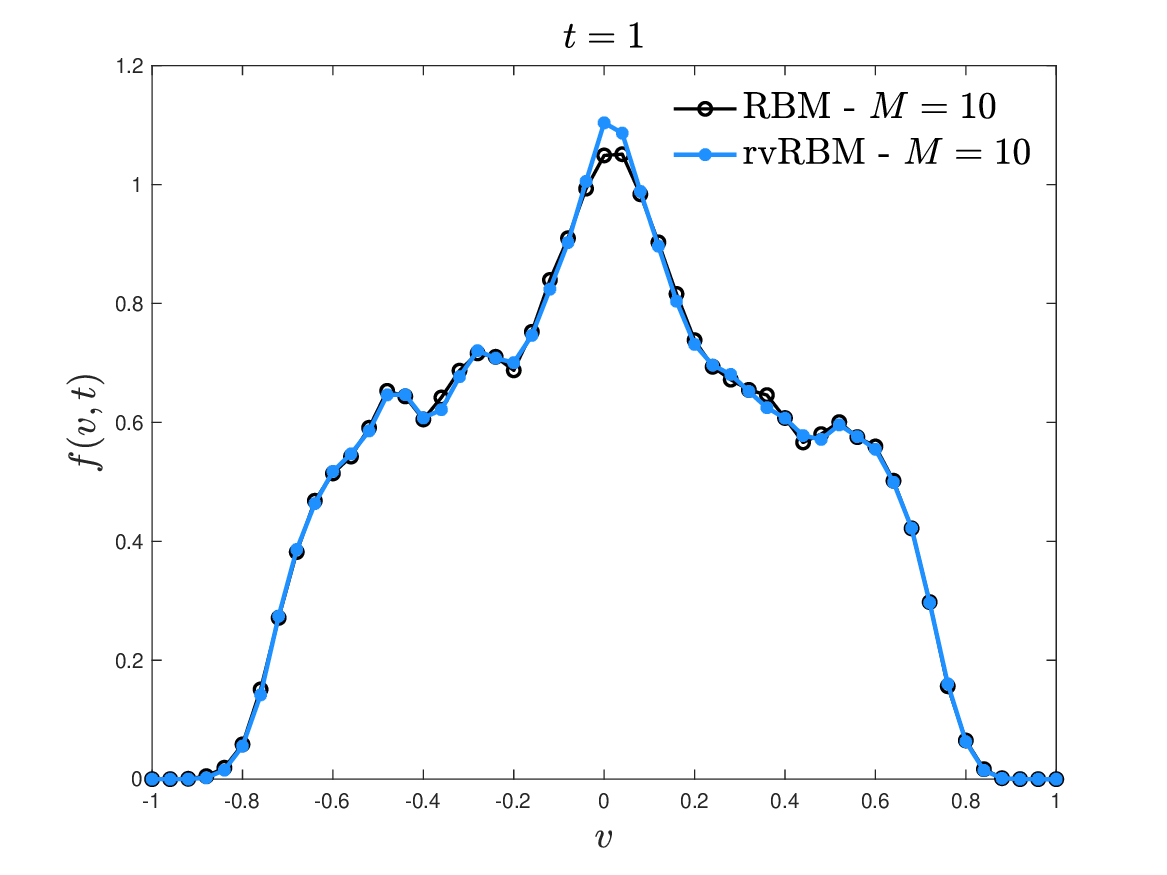} }
\subfigure[$t=5,M=5$]{
\includegraphics[scale = 0.17]{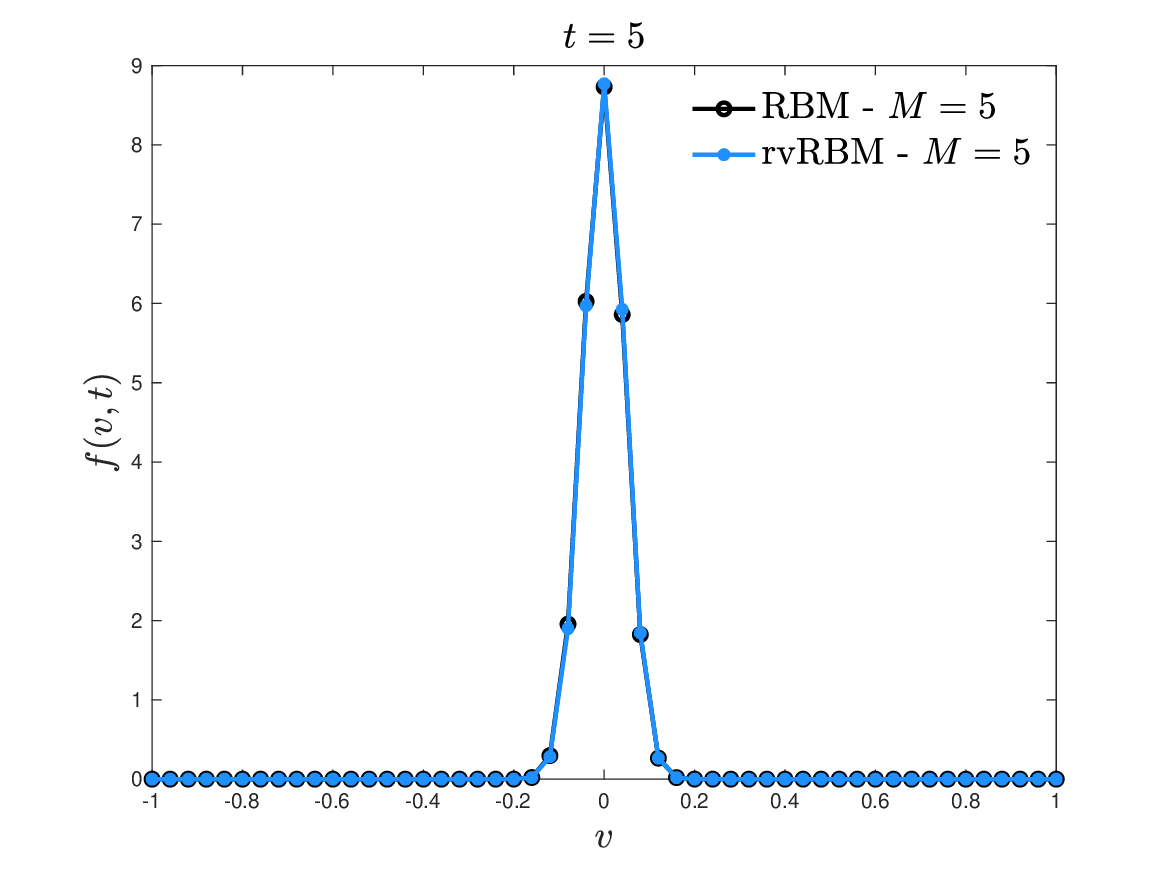}}
\subfigure[$t=5,M=10$]{
\includegraphics[scale = 0.17]{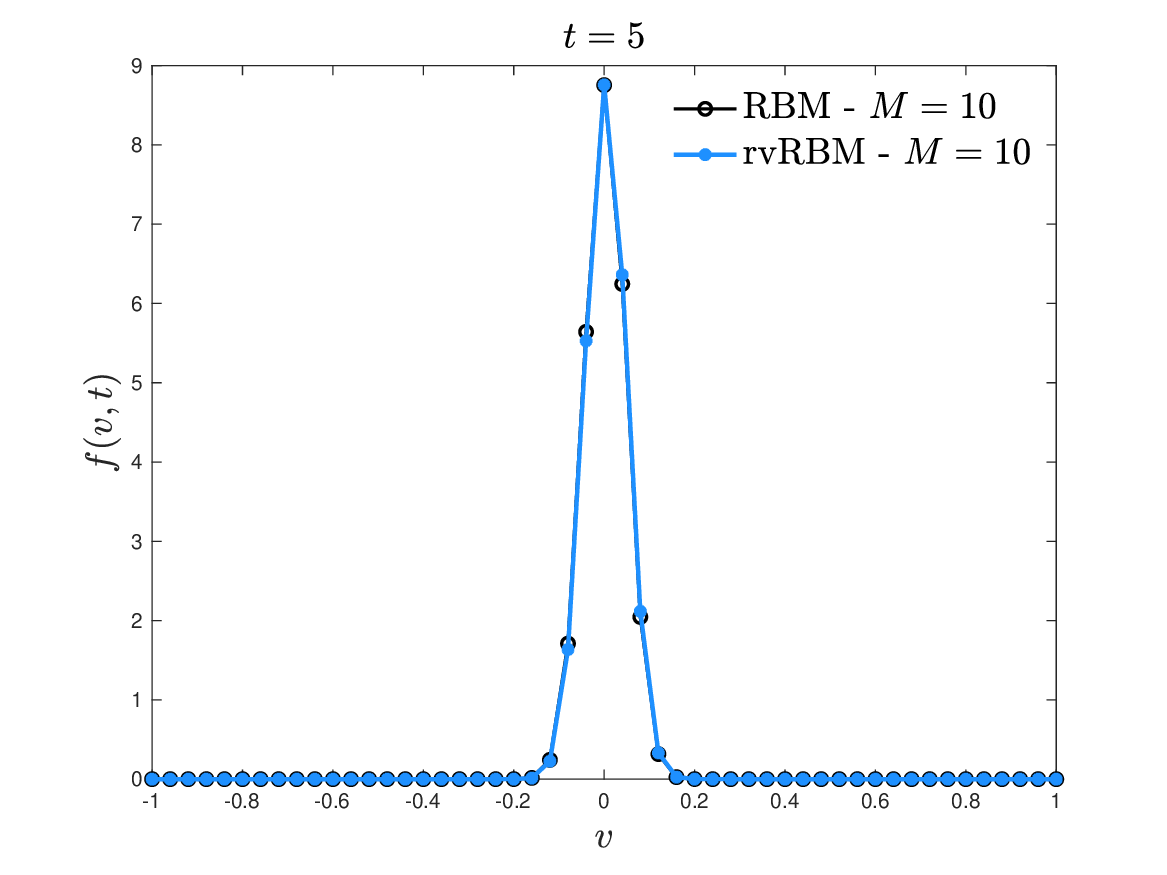}}\\
\subfigure[$t=1,M=5$]{
\includegraphics[scale = 0.17]{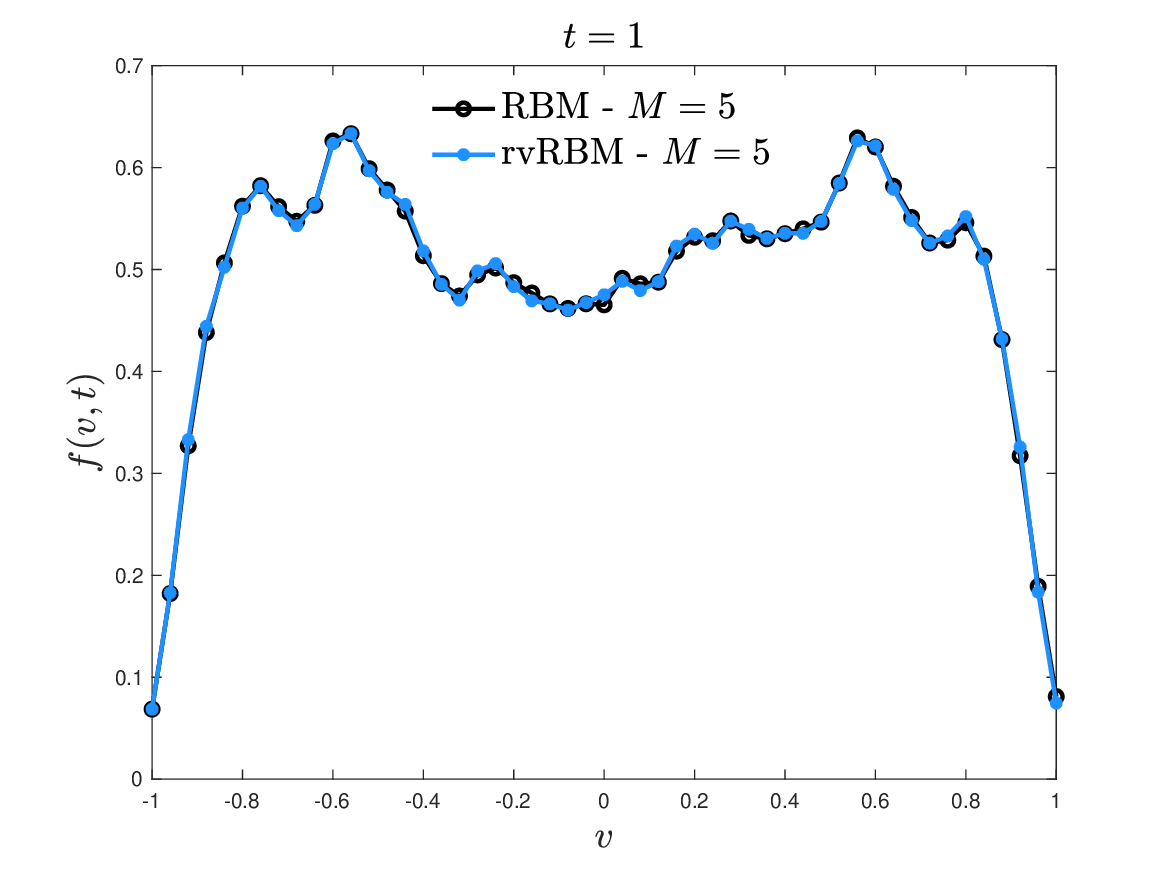}}
\subfigure[$t=1,M=10$]{
\includegraphics[scale = 0.17]{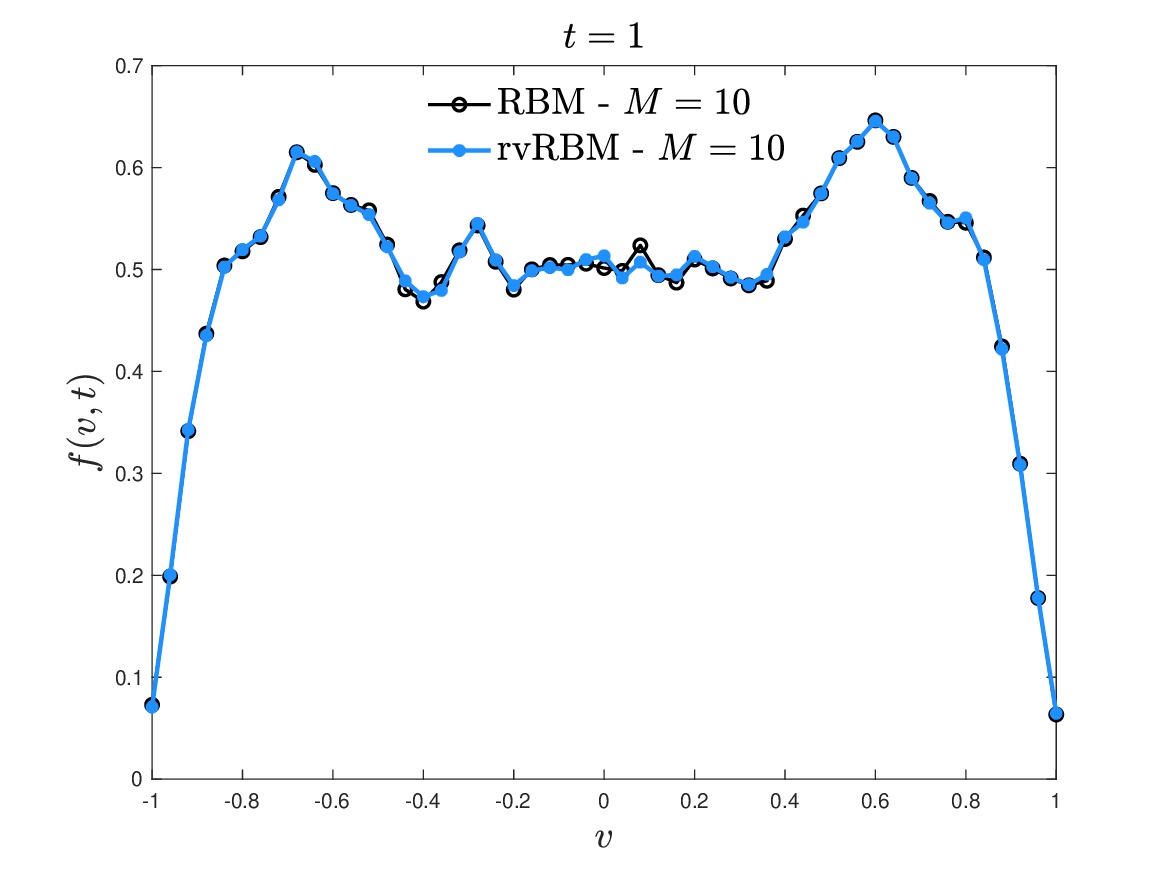} }
\subfigure[$t=5,M=5$]{
\includegraphics[scale = 0.17]{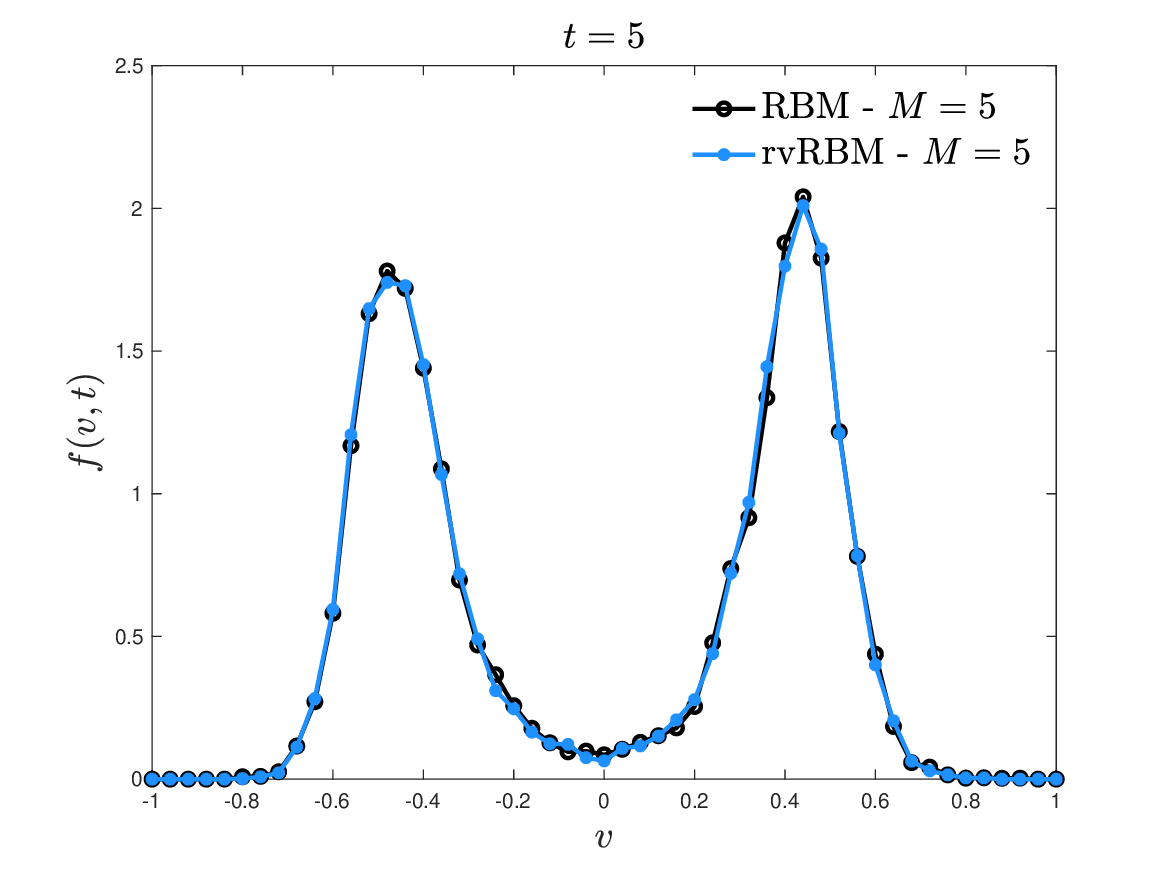}}
\subfigure[$t=5,M=10$]{
\includegraphics[scale = 0.17]{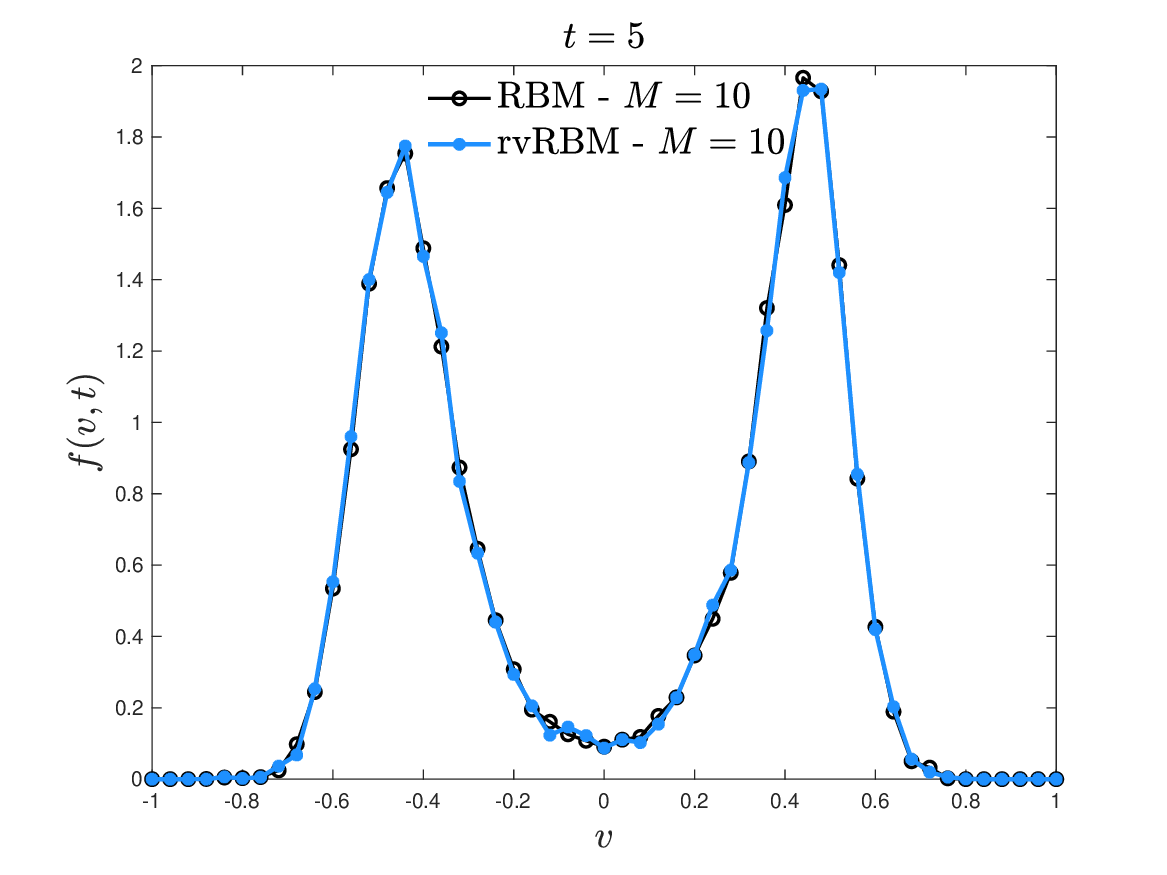}}
\caption{\textbf{Test 2.} Evolution of the densities for the stochastic bounded confidence model \eqref{eq:BCstoch} with interaction function \eqref{eq:BCinter} and $\delta = 1$ (top row), $\delta = 0.5$ (bottom row). We considered $N = 10^5$ particles with initial distribution \eqref{eq:f0BC} and we compare the evolution of the approximated densities obtained with RBM or rvRBM with a subset of interacting particles of size $M = 5$ or $M = 10$. We fixed the diffusion constant $\sigma^2 = 10^{-1}$.  }
\label{fig:BC_stoch}
\end{figure}

In Figure \ref{fig:error_BCstoch} we show the evolution in time of the introduced absolute error in the case $\delta = 1$ (top row) obtained through the standard RBM and the introduced rvRBM method. We may observe how rvRBM leads to a substantial advantage if the emerging distributions are highly correlated, as in the case $\delta = 1$. At variance with the deterministic model, the asymptotic level of accuracy of rvRBM is influenced by the total number of particles due to the presence of the stochastic component in the dynamics. In the second row we report the error produced by a fixed batch of size $M = 5$ (left) or $M = 10$ (right) and an increasing total number of particles $N \in \{10,\dots,10^4\}$. The error is computed at time $T = 5$. We may observe how the reduced variance approach is capable to achieve higher accuracy with the same number of particles, and therefore of cost, of the RBM method. 

\begin{figure}
\centering
\includegraphics[scale = 0.35]{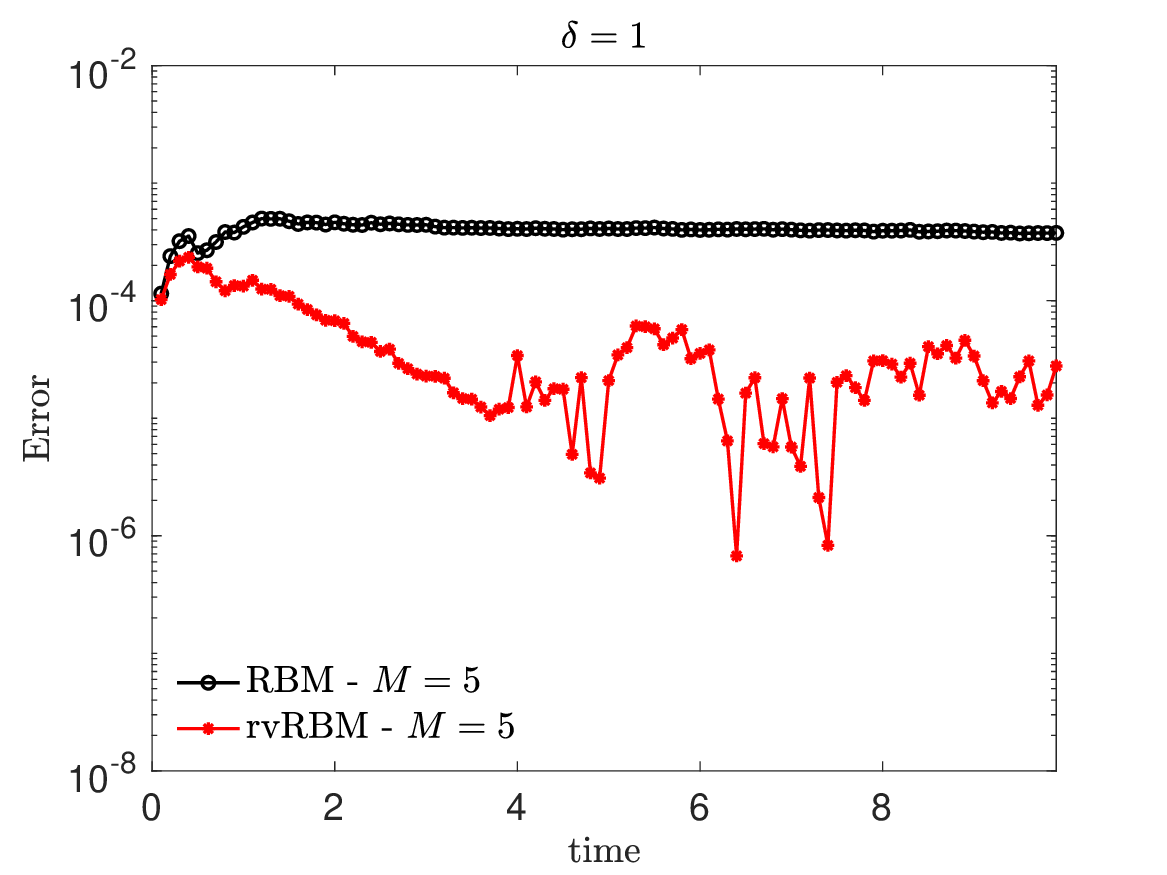}
\includegraphics[scale = 0.35]{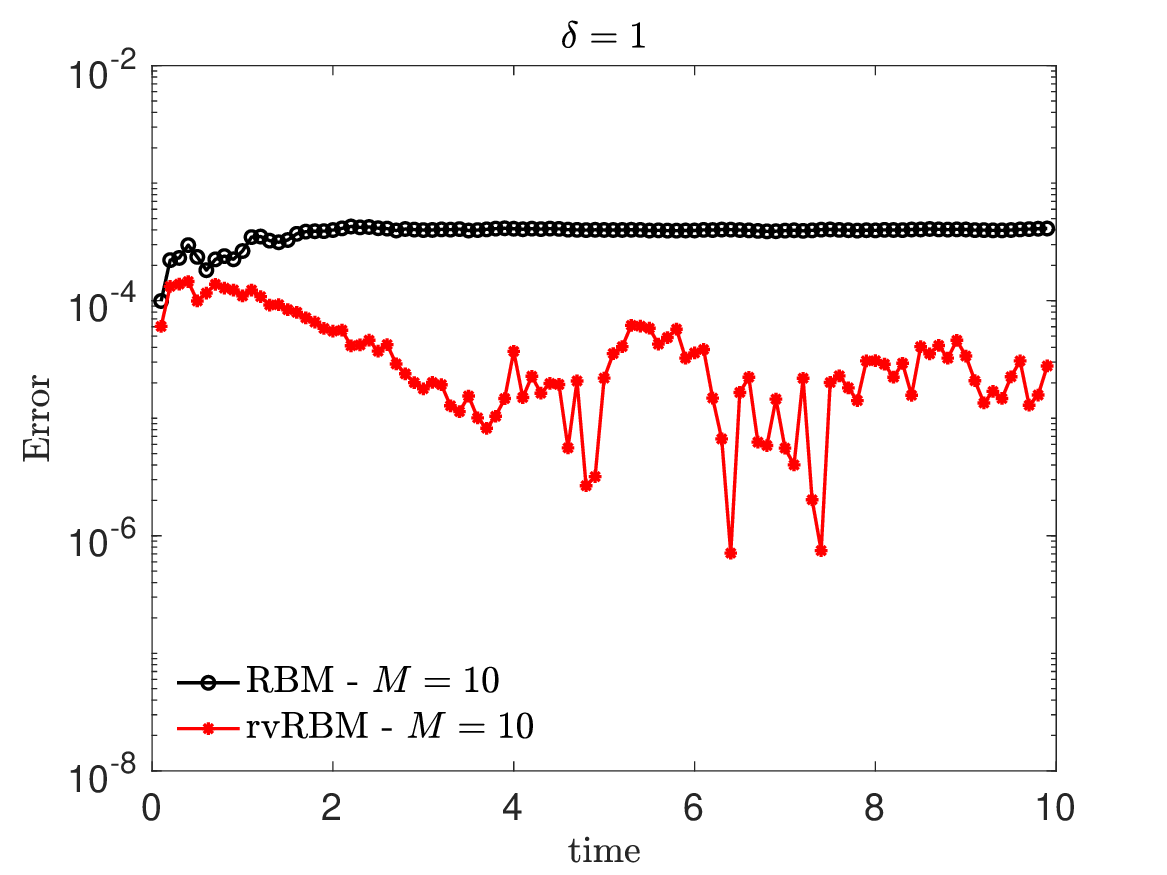} \\
\includegraphics[scale = 0.35]{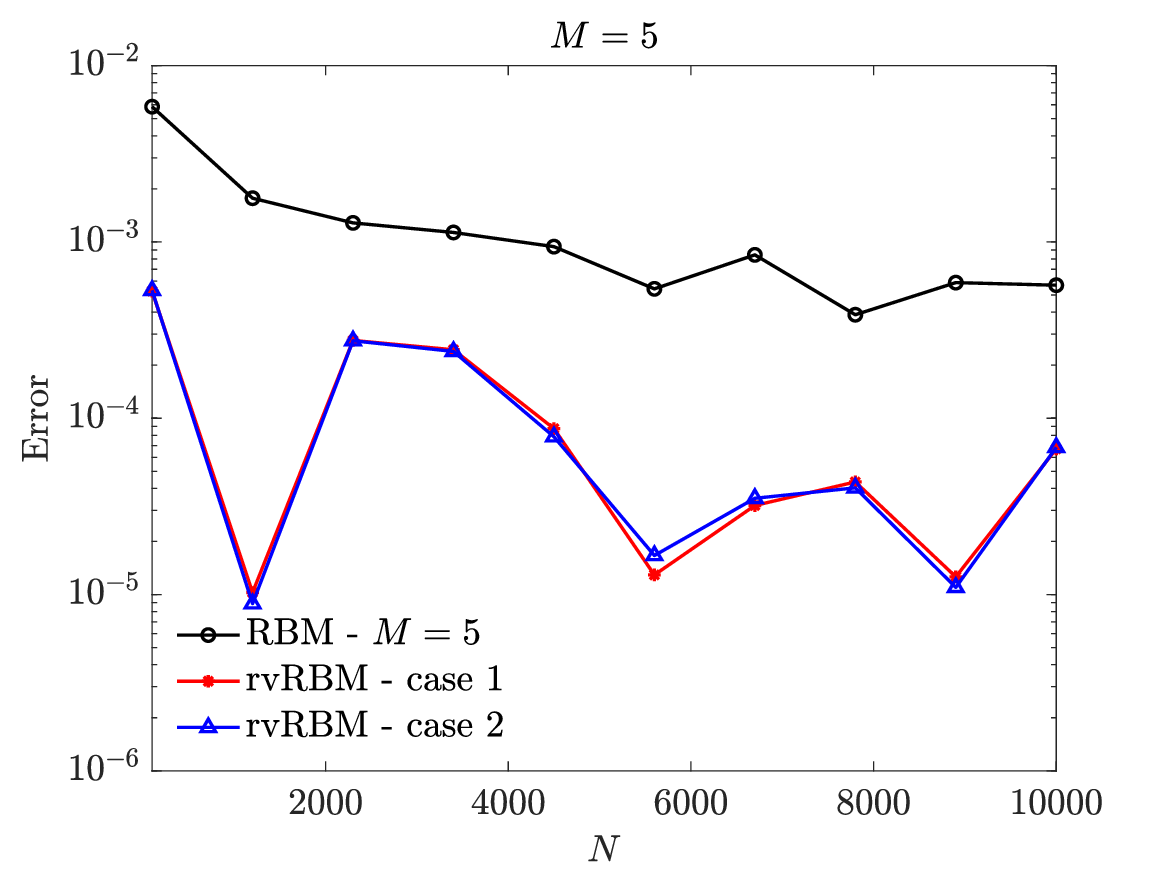}
\includegraphics[scale = 0.35]{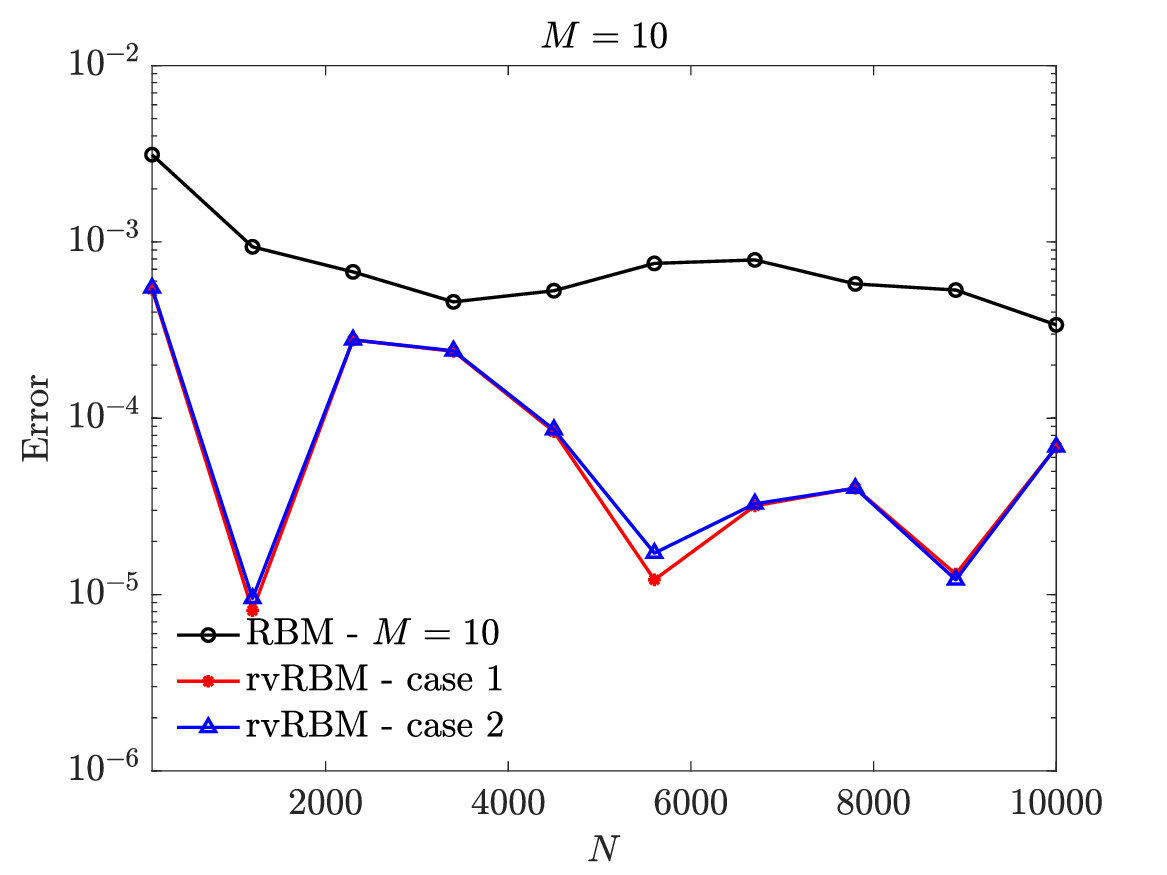} 
\caption{\textbf{Test 2.} Stochastic bounded confidence model.  Top row: evolution of the absolute error defined in \eqref{eq:error} for the stochastic bounded confidence model solved through RBM and rvRBM with $M = 5$ (left) or $M=10$ (right). Bottom row: errors of RBM and rvRBM for a fixed batch size $M = 5$ (left) or $M = 10$ (right) and variable size of the sample $N \in \{10,\dots,10^4\}$. In all the presented tests we considered $\delta = 1$ and $N = 10^5$ particles. }
\label{fig:error_BCstoch}
\end{figure}

\subsection{Test 3: Cucker-Smale model}
We consider the second order model for swarming introduced in Section \ref{sect:swarming}.  We recall that it is possible to prove that this model corresponds to the meanfield limit of the following particles' system
\[
\begin{split}
\dfrac{d}{dt} x_i &= v_i, \\
\dfrac{d}{dt} v_i &= \dfrac{1}{N} \sum_{j=1}^N P(|x_i-x_j|) (v_j-v_i), \qquad i = 1,\dots,N,
\end{split}
\]
where $P(|x_i-x_j|) = \frac{1}{(\xi^2 + |x_i-x_j|^2)^\beta}$, we will always consider  $\xi = 1$ and $\beta = 10^{-1}$, for which unconditional alignment emerges asymptotically. Hence, the RBM algorithm corresponds to consider 
\begin{equation}
\label{eq:CS_RBM}
\begin{split}
\dfrac{d}{dt} x_i &= v_i, \\
\dfrac{d}{dt} v_i &= \dfrac{1}{M} \sum_{j\in \mathcal S_M} \dfrac{v_j-v_i}{(\xi^2 + |x_i-x_j|^2)^\beta}, \qquad i = 1,\dots,N,
\end{split}
\end{equation}
where $\mathcal S_M$ is a uniform subsample of size $M>1$ obtained from $\{(x_i,v_i)\}_{i=1}^N \in \mathbb R^{2d}$. In the following, we construct a rvRBM method by considering 
\begin{equation}
\label{eq:CS_rvRBM}
\begin{split}
\dfrac{d}{dt} x_i &= v_i, \\
\dfrac{d}{dt} v_i &= \dfrac{1}{M} \sum_{j\in \mathcal S_M} \dfrac{v_j-v_i}{(\xi^2 + |x_i-x_j|^2)^\beta} - \lambda^*(x_i,v_i)\tilde P(x_i,v_i)(u_M-u_N), \qquad i = 1,\dots,N,
\end{split}
\end{equation}
and $\tilde P \equiv 1$, being 
\[
u_M = \dfrac{1}{M}\sum_{j\in \mathcal S_M}v_i, \qquad u_N = \dfrac{1}{N}\sum_{j=1}^Nv_i,
\]
where $u_N(t) \equiv u_N(0)$ being $P(|x_i-v_i|)$ a symmetric interaction function. Hence, the optimal value of $\lambda(x_i,v_i)$ is computed as in \eqref{eq:rvRBM}. 

In Figure \ref{fig:CS_distrib} we show the evolution of the reconstructed distribution functions in the phase space obtained with RBM (left panels) and rvRBM (right panels). We considered $N = 10^5$ particles having uniform initial distribution in both space and velocity, i.e.
\begin{equation}
\label{eq:CSf0}
f_0(x,v) = 
\begin{cases}
\frac{1}{4} &  (x,v) \in [-1,1]\times [-1,1] \\
0 & \textrm{elsewhere.}
\end{cases}
\end{equation}
Hence, we considered a RBM/rvRBM methods fixing a subsampling of size $M = 10$. We may observe how the obtained distributions consistent between the two methods also at the level of marginal distributions (bottom row).  

\begin{figure}
\centering
\includegraphics[scale = 0.35]{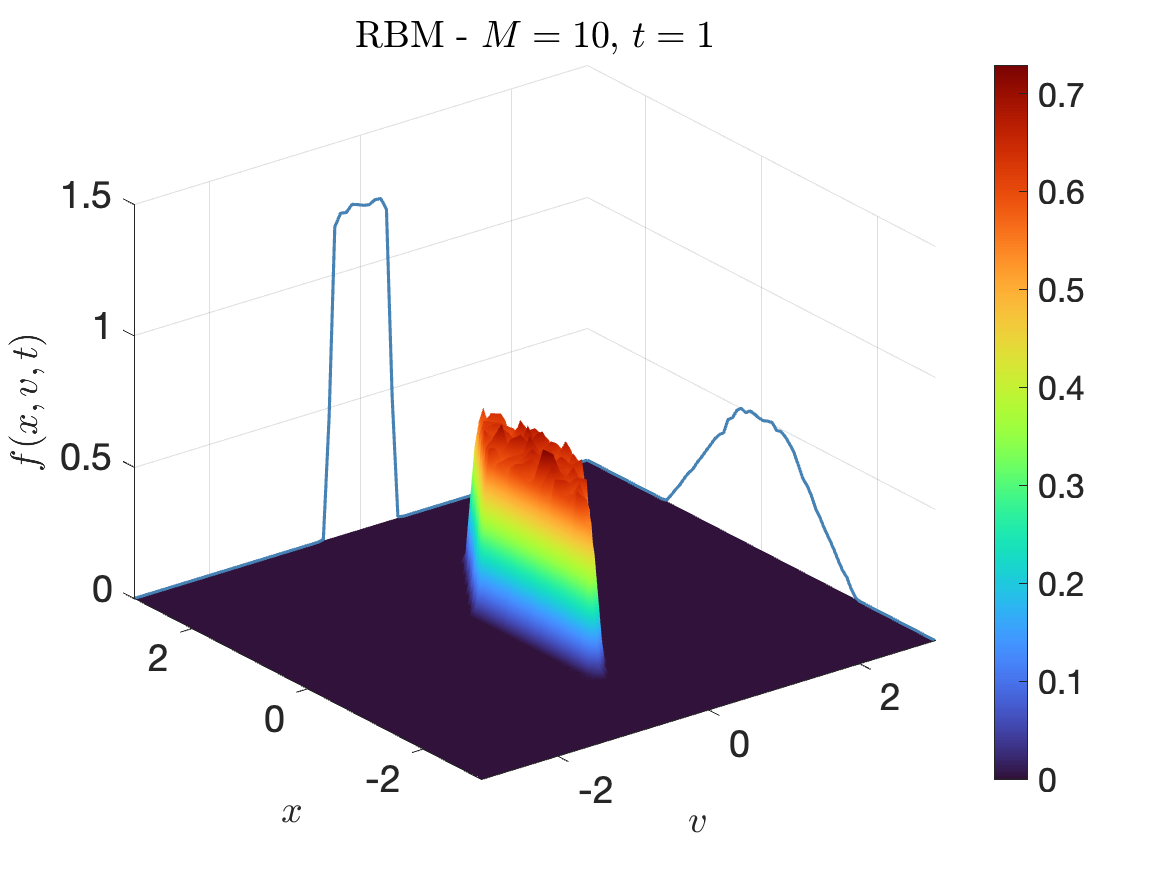}
\includegraphics[scale = 0.35]{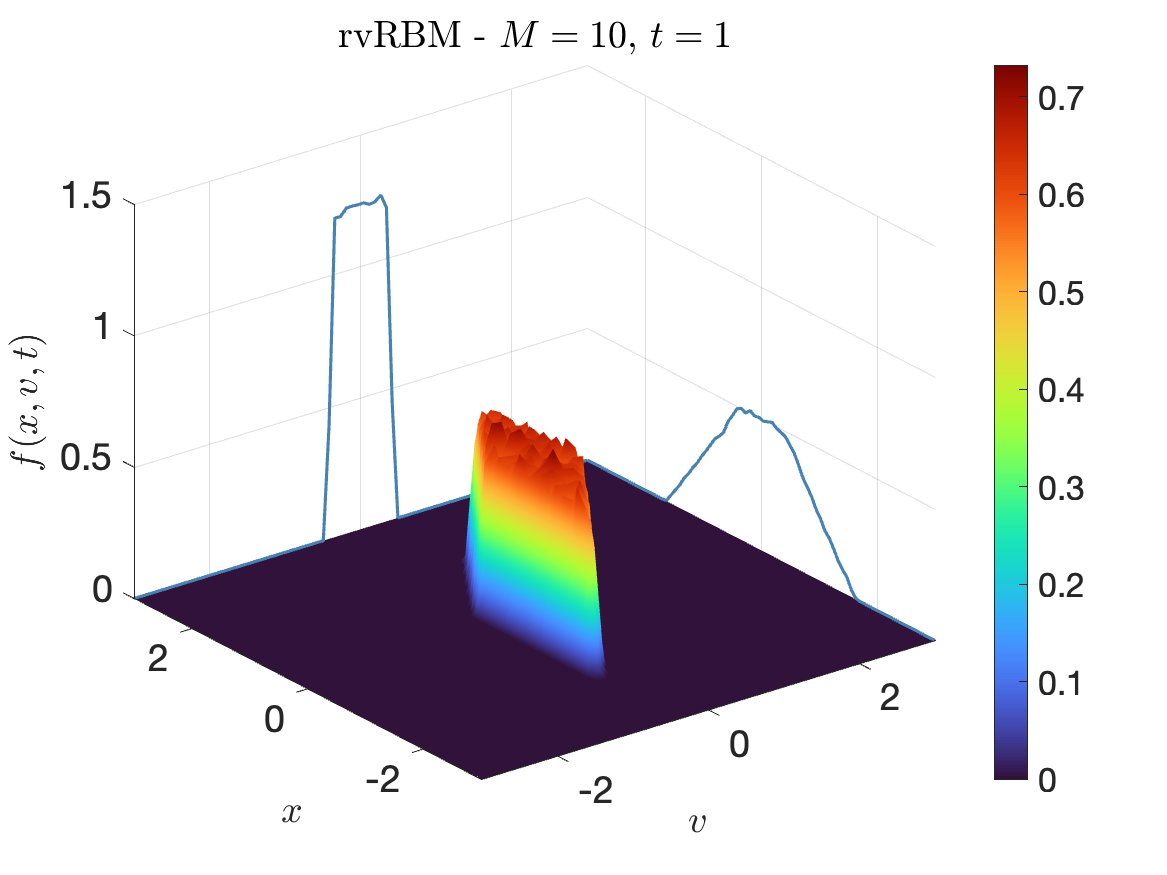}\\
\includegraphics[scale = 0.35]{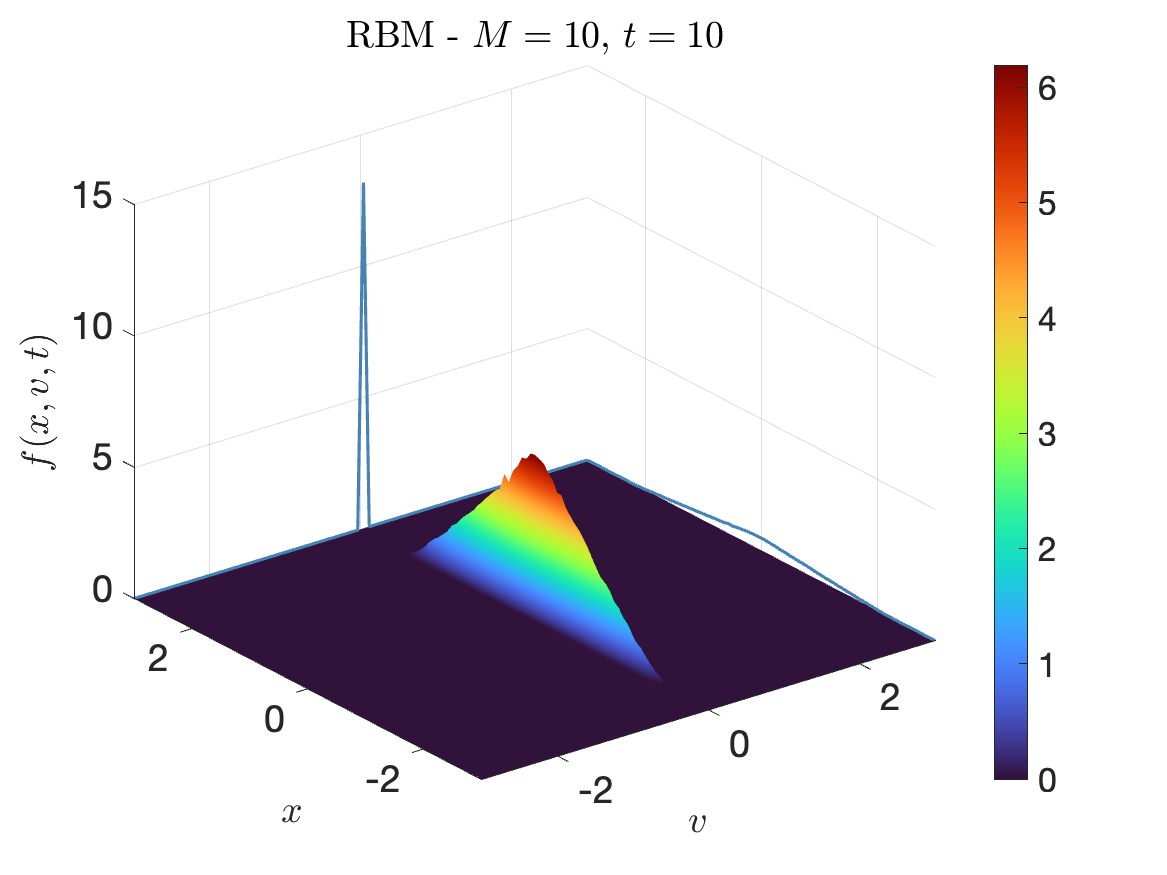}
\includegraphics[scale = 0.35]{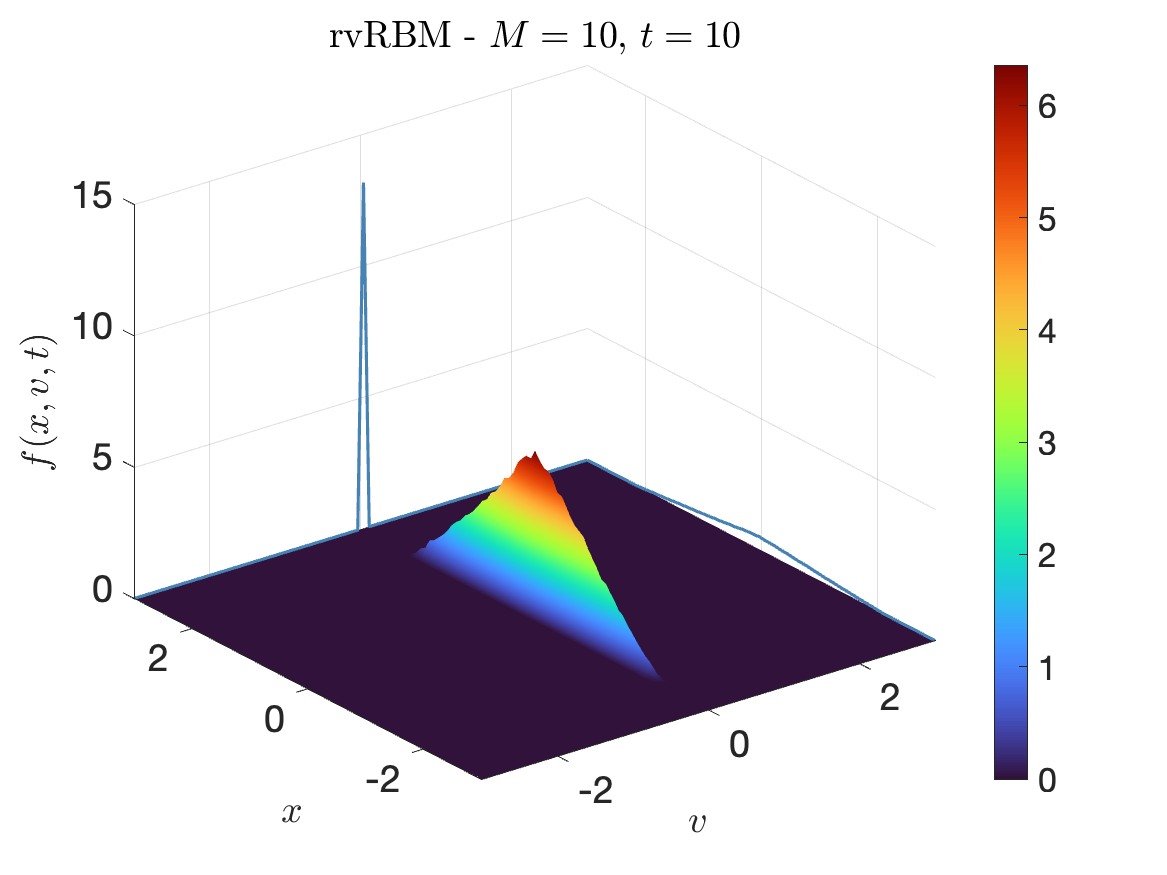}\\
\includegraphics[scale = 0.35]{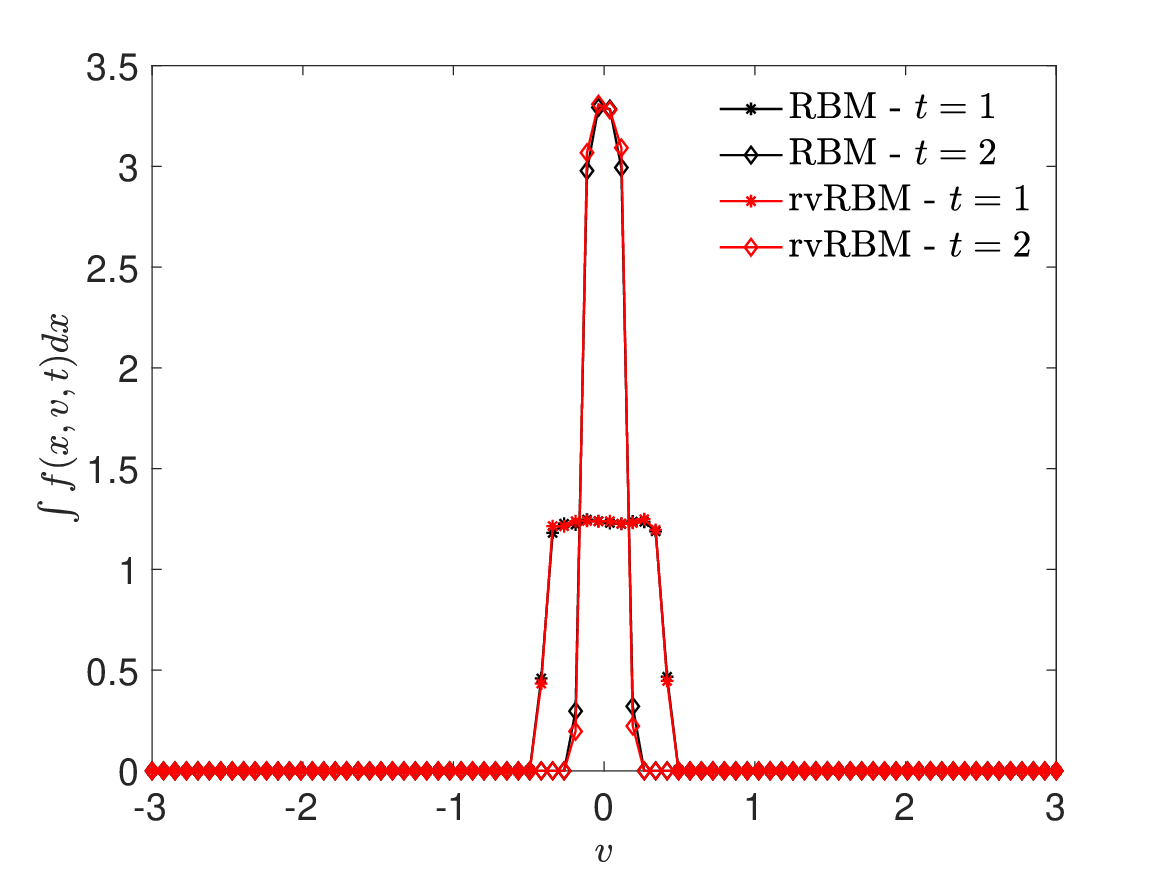}
\includegraphics[scale = 0.35]{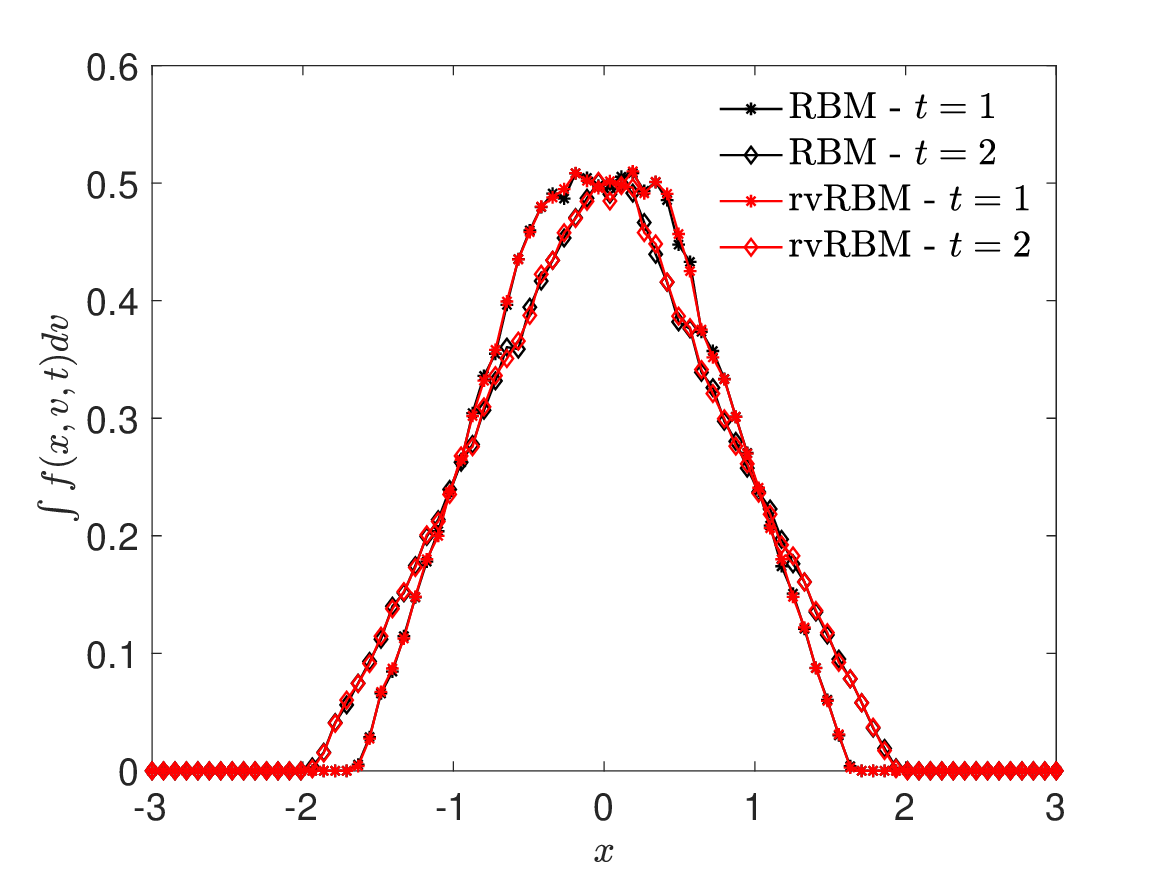}
\caption{\textbf{Test 3.} Evolution of the kinetic distributions obtained with RBM (left) or rvRBM (right) at time $t = 1$ (top row) and $t = 10$ (bottom row). For both methods we considered a subsample of size $M = 10$. In the bottom row we report the marginal densities obtained at time $t = 1$ and $t = 2$. The initial size of the particle system is $N = 10^5$ with initial distribution as in \eqref{eq:CSf0}. }
\label{fig:CS_distrib}
\end{figure}

We report in Figure \ref{fig:errorCS} the evolution of the error obtained with RBM and rvRBM. In detail, we fixed $N  = 10^5$ initial particles and a subsample of size $M = 5$ (left plot) or $M = 10$ (right plot). In both cases, we can observe the advantages produced by the reduced variance strategy. In Figure \ref{fig:lambdaCS} we report the evolution of the average optimal $\lambda^*$ computed for the Cucker-Smale model with either $M = 5$ or $M = 10$. We can observe how, due to the transport term the surrogate model lose the initial correlation with the full model. 

\begin{figure}
\centering
\includegraphics[scale = 0.35]{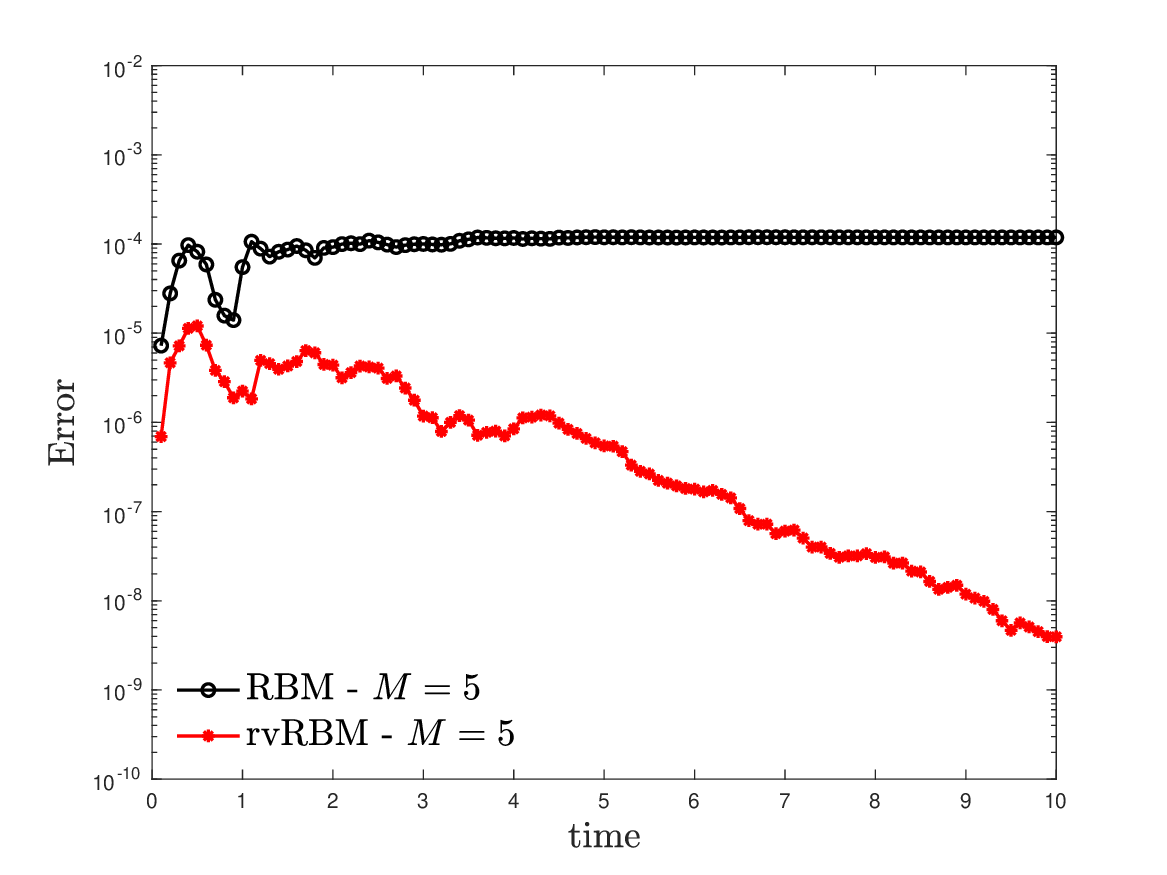}
\includegraphics[scale = 0.35]{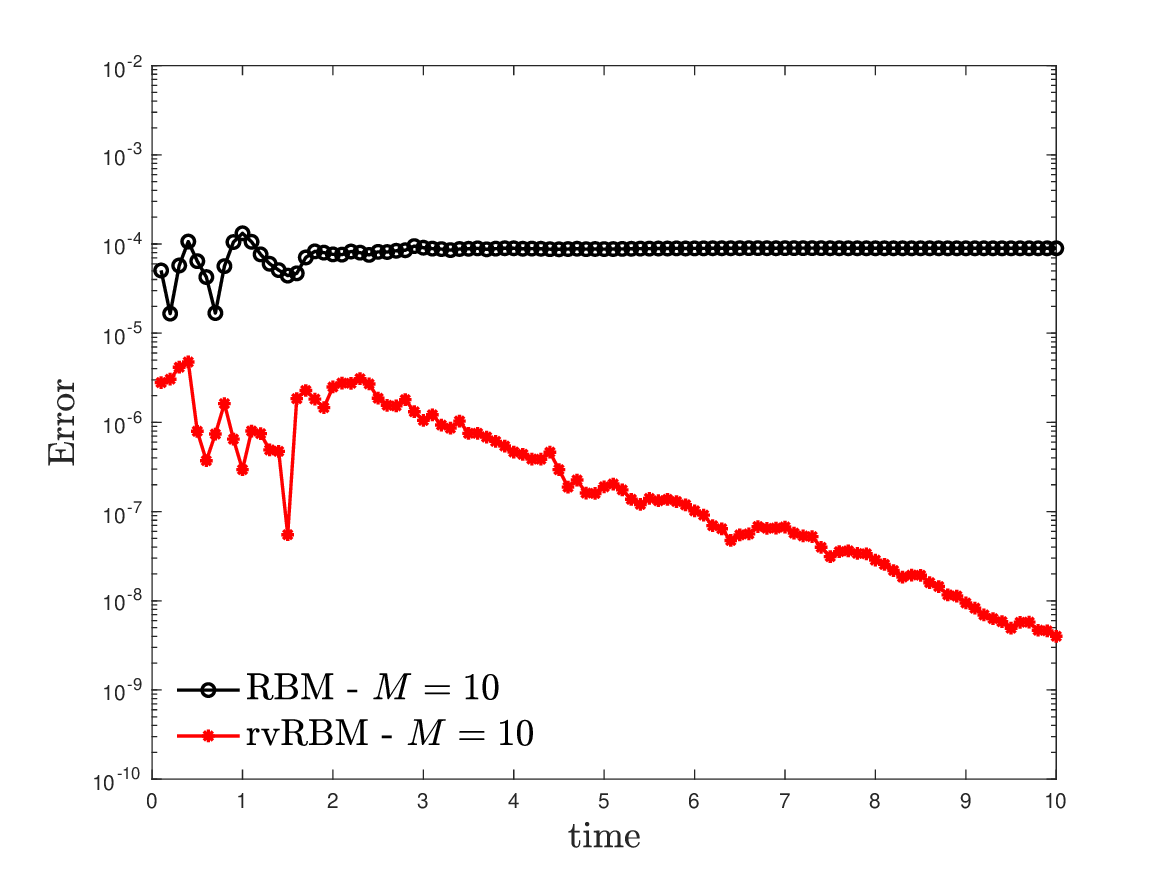}
\caption{\textbf{Test 3.} Evolution of the absolute error for the Cucker-Smale model solved through RBM as in \eqref{eq:CS_RBM} and rvRBM as in \eqref{eq:CS_rvRBM} with $M = 5$ (left) or $M = 10$ (right). }
\label{fig:errorCS}
\end{figure}

\begin{figure}
\centering
\includegraphics[scale = 0.35]{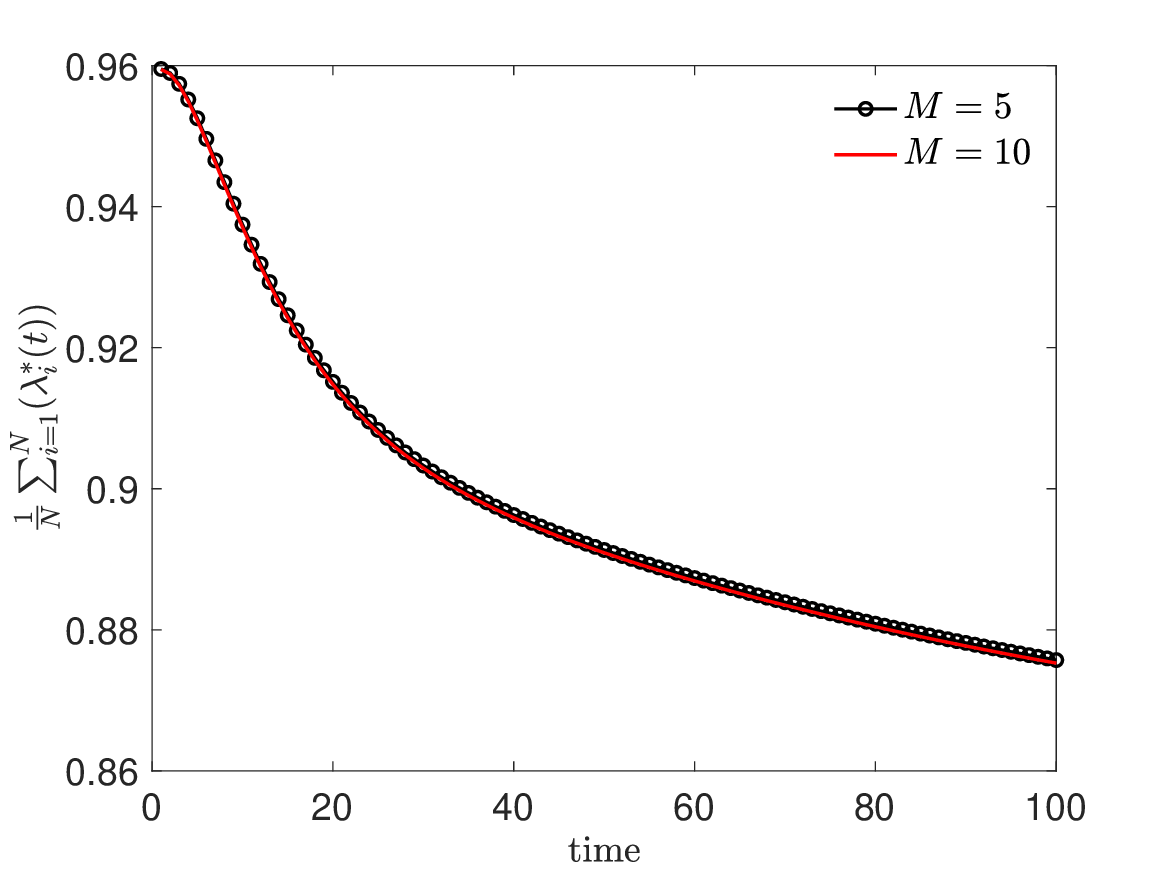}
\caption{\textbf{Test 3.} Evolution of the average optimal weight $\frac{1}{N} \sum_{i=1}^N \lambda_i^*(t)$ for the Cucker-Smale model approximated through rvRBM strategy with either $M = 5$ or $M = 10$. }
\label{fig:lambdaCS}
\end{figure}

\section{Conclusions}
The computational challenges associated with the $O(N^2)$ complexity of particle models for large systems with nonlocal interactions have driven the adoption of Random Batch Methods (RBM). While these methods, using batches of size $1<M<N$, efficiently reduce the computational burden to $O(MN)$, introducing an additional error becomes an inevitable trade-off. Addressing this challenge, our paper proposes a control variate strategy inspired by ideas from uncertainty quantification. By leveraging a reduced complexity surrogate model, we successfully mitigate batch variance, achieving a substantial reduction and, consequently, an improved computational cost-to-accuracy ratio.

Of course, the effectiveness of our proposed strategy is strongly dependent on the possibility to construct a surrogate model which keeps correlations with the original model. Nevertheless, even in the worst case scenario when the surrogate model looses correlations our methodology will never deteriorate its accuracy compared to standard RBM methods. This is demonstrated through numerical examples, particularly in the context of models depicting collective behavior in opinion spreading and swarming dynamics. This not only showcases the practical utility of our approach but also underlines its versatility across different scenarios.

Looking ahead, our work opens directions for further research, particularly in refining control variate strategies and optimizing the implementation of Random Batch Methods in various computational contexts. For example, one could also explore the use of a multilevel strategy as an alternative to a surrogate model^^>\cite{G}.
In this sense, the present approach marks a promising step towards bridging the gap between the microscopic complexity of interacting particle systems and the computational feasibility required for meaningful simulations, with potential implications for a broad spectrum of applications. 

\section*{Acknowledgements}
This research has been written within the activities of the national groups
GNCS (Gruppo Nazionale per il Calcolo Scientifico) and GNFM (Gruppo Nazionale di Fisica Matematica) of INdAM, Italy. MZ acknowledges the support of MUR-PRIN2020 Project No.2020JLWP23 (Integrated Mathematical Approaches to Socio-Epidemiological Dynamics). The research of LP has been supported by the Royal Society under the Wolfson Fellowship ``Uncertainty quantification, data-driven simulations and learning of multiscale complex systems governed by PDEs''. The work has been partially supported by the National Centre for HPC, Big Data and Quantum Computing (CN00000013). The authors also acknowledge the support of the Banff International Research Station (BIRS) for the Focused Research Group (22frg198) “Novel perspectives in kinetic equations for emerging phenomena”, July 17–24, 2022, where part of this work was done.


\begin{thebibliography}{99}

\bibitem{AP}
G. Albi, L. Pareschi. Binary interaction algorithms for the simulation of flocking and swarming dynamics. \emph{Multiscale Model. Simul.}, \textbf{11}:1--29, 2013. 

\bibitem{APZ17} 
G. Albi, L. Pareschi, M. Zanella. Opinion dynamics over complex networks: kinetic modelling and numerical methods. \emph{Kinet. Relat. Mod.}, \textbf{10}(1): 1--32, 2017. 

\bibitem{BR}
V. Barbu, M. Röckner. From nonlinear Fokker-Planck equations to solutions of distribution dependent SDE. \emph{Ann. Prob.}, \textbf{48}:1902--1920, 2020. 

\bibitem{BL}
D. Borra, T. Lorenzi. Asymptotic analysis of continuous opinion dynamics models under bounded confidence. \emph{Commun. Pure Appl. Anal.}, \textbf{12}:1487--1499, 2013. 

\bibitem{Caf}
R. E. Caflisch. Monte Carlo and quasi-Monte Carlo methods, \emph{Acta Numer.}, \textbf{7}:1--49,1998.

\bibitem{CFRT}
J. A. Carrillo, M. Fornasier, J. Rosado, G. Toscani. Asymptotic flocking dynamics for the kinetic Cucker-Smale model, \emph{SIAM J. Math. Anal.}, \textbf{42}(1):218--236, 2010. 

\bibitem{CFTV}
J. A. Carrillo, M. Fornasier, G. Toscani, F. Vecil. Particle, kinetic, and hydrodynamic models of swarming. In \emph{Mathematical Modeling of Collective Behavior in Socio-Economic and Life Sciences}, Birkhäuser Boston, 2010. 

\bibitem{CJT}
J. A. Carrillo, S. Jin, Y. Tang. Random batch particle methods for the homogeneous Landau equation, \emph{Commun. Comput. Phys.}, \textbf{31}:997--1019, 2022. 

\bibitem{CKR}
J. A. Carrillo, A. Klar, A. Roth. Single to double mill small noise transition via semi-Lagrangian finite volume methods. \emph{Comm. Math. Sci.}, \textbf{14}:1111--1136, 2016.

\bibitem{CPZ}
J. A. Carrillo, L. Pareschi, M. Zanella. Particle based gPC methods for mean-field models of swarming with uncertainty. \emph{Commun. Comput. Phys.}, \textbf{25}:508--531, 2019.

\bibitem{CT}
J. A. Carrillo,  G. Toscani. Contractive probability metrics and asymptotic behavior of dissipative kinetic equations. \emph{Riv. Mat. Univ. Parma}, \textbf{6}(7):75--198, 2007.

\bibitem{CZ}
J. A. Carrillo, M. Zanella. Monte Carlo gPC methods for diffusive kinetic flocking models with uncertainties. \emph{Vietnam J. Math.}, \textbf{47}(4): 931--954, 2019. 

\bibitem{CHY}
Y.-P. Choi, S.-Y. Ha, S.-B. Yun. Complete synchronization of Kuramoto oscillators with finite inertia. \emph{Phys. D}, \textbf{240}:32--44, 2011. 

\bibitem{CS}
F. Cucker, S. Smale. Emergent behavior in flocks. \emph{IEEE Trans. Automat. Control}, \textbf{52}:852--862, 2007.

\bibitem{DP}
G. Dimarco, L. Pareschi. Multi-scale control variate methods for uncertainty quantification in kinetic equations. \emph{J. Comput. Phys.}, \textbf{388}:63--89, 2019.

\bibitem{DP2020}
G. Dimarco, L. Pareschi. Multiscale variance reduction methods based on multiple control variates for kinetic equations with uncertainties. \emph{Multiscale Model. Simul.}, \textbf{18}(1):351--382, 2020. 

\bibitem{DO}
M. R. D'Orsogna, Y. L. Chuang, A. L. Bertozzi, L. Chayes. Self-propelled particles with soft-core interactions: patterns, stability and collapse. \emph{Phys. Rev. Lett.}, \textbf{96}(10):104302, 2006. 

\bibitem{DW}
B. Düring, P. Markowich, J. F. Pieschmann, M.-T. Wolfram. Boltzmann and Fokker-Planck equations modelling opinion formation in the presence of strong leaders. \emph{Proc. R. Soc. A}, \textbf{465}:3687--3708, 2009. 

\bibitem{FK}
A. Figalli and M.-J. Kang. A rigorous derivation from the kinetic Cucker-Smale model to the pressureless Euler system with nonlocal alignment, \emph{Anal. PDE},\textbf{12}(3):843--866, 2019.

\bibitem{FPTT}
G. Furioli, A. Pulvirenti, E. Terraneo, G. Toscani. Fokker-Planck equations in the modeling of socio-economic phenomena. \emph{Math. Mod. Meth. Appl. Sci.}, \textbf{27}(01):115--158, 2017. 

\bibitem{GP}
 J. Garnier, G. Papanicolaou,  T.- W. Yang. Consensus convergence with stochastic effects. \emph{Vietnam J. Math.}, \textbf{45}(1-2):51--75, 2017.
 
\bibitem{G}
M. B. Giles. Multilevel Monte Carlo methods. \emph{Acta Numer.}, \textbf{24}:259--328, 2015. 


\bibitem{HL}
S.-Y. Ha, J.-G. Liu. A simple proof of the Cucker-Smale flocking dynamics and mean-field limit. \emph{Commun. Math. Sci.}, \textbf{7}:297--325, 2009.

\bibitem{HT}
S.-Y. Ha, E. Tadmor. From particle to kinetic and hydrodynamic descriptions of flocking. \emph{Kinet. Relat. Models}, \textbf{1}(3):415--435, 2008. 

\bibitem{HK}
R. Hegselmann, U. Krause. Opinion dynamics and bounded confidence: models, analysis and simulation. \emph{J. Artif. Soc. Soc. Simul.}, \textbf{5}(3), 2022. 

\bibitem{HPV} M. Herty, L. Pareschi, G. Visconti. Mean field models for large data–clustering problems. \emph{Netw. and Heter. Media}, \textbf{15}(3): 463--487, 2020. 

\bibitem{Ja}
P. E. Jabin. A review of the mean field limits for Vlasov equations. \emph{Kinet. Relat. Models}, \textbf{7}: 661--711, 2020. 

\bibitem{JL}
S. Jin, L. Li. On the mean field limit of the Random Batch Method for interacting particle systems. \emph{Sci. China Math.}, \textbf{65}:169--202, 2022. 

\bibitem{JLL}
S. Jin, L. Li, J.-G. Liu. Random Batch Methods (RBM) for interacting particle systems. \emph{J. Comput. Phys.}, \textbf{400}:108877, 2020. 

\bibitem{JLLq}
S. Jin, L. Li. Random Batch Methods for Classical and Quantum Interacting Particle Systems and Statistical Samplings. In: Bellomo, N., Carrillo, J.A., Tadmor, E. (eds) \emph{Active Particles, Volume 3. Modeling and Simulation in Science, Engineering and Technology}. Birkhäuser, Cham, 2022.

\bibitem{KZ} D. Ko, E. Zuazua. Model predictive control with random batch methods for a guiding problem. \emph{Mathematical Models and Methods in Applied Sciences} \textbf{31}(8):1569--1592, 2021. 

\bibitem{MT}
S. Motsch, E. Tadmor. Heterophilious dynamics enhances consensus. \emph{SIAM Rev.}, \textbf{56}(4):577--621, 2014. 

\bibitem{Pa21}
L. Pareschi. An Introduction to Uncertainty Quantification for Kinetic Equations and Related Problems. In: Albi, G., Merino-Aceituno, S., Nota, A., Zanella, M. (eds) \emph{Trails in Kinetic Theory. SEMA SIMAI Springer Series} \textbf{25}:141--181, 2021. 

\bibitem{PTTZ}
L. Pareschi, G. Toscani, A. Tosin, M. Zanella. Hydrodynamic models of preference formation in multi-agent societies. \emph{J. Nonlin. Sci.}, \textbf{29}(6):2761--2796, 2019. 

\bibitem{PZ0}
L. Pareschi, M. Zanella.  Structure preserving schemes for nonlinear Fokker-Planck equations and applications. \emph{J. Sci. Comput.}, \textbf{74}(3): 1575-1600, 2018.


\bibitem{CBO} 
R. Pinnau, C. Totzeck, O. Tse, S. Martin. A consensus-based model for global optimization and its mean-field limit. \emph{Math. Mod. and Meth. Appl. Scie.}, \textbf{27}(1):183--204, 2017. 

\bibitem{ST}
R. Shu, E. Tadmor. Anticipation breeds alignment, \emph{Arch. Rational Mech. Anal. }, \textbf{240}:203--241, 2021. 

\bibitem{T0}
G. Toscani. Entropy production and the rate of convergence to equilibrium for the Fokker-Planck equation, \emph{Quart. Appl. Math.}, \textbf{57}:521--541, 1999. 

\bibitem{T}
G. Toscani. Kinetic models of opinion formation. \emph{Commun. Math. Sci.}, \textbf{4}(3):481--496, 2006. 

\bibitem{Villani}
C. Villani. A review of mathematical topics in collisional kinetic theory. \emph{Handbook of Mathematical Fluid Dynamics}, \textbf{1}(71-305): 3--8, 2002.

\end{thebibliography}
\end{document}